\documentclass[12pt]{amsart}
\usepackage[utf8]{inputenc}

\usepackage{amsmath,amsthm,amssymb,mathscinet,bbm}
\usepackage{parskip}
\usepackage{geometry}
\usepackage{xcolor} 
\usepackage{xfrac, nicefrac}
\usepackage{exscale, relsize}

\usepackage{mathtools} 
\usepackage[hidelinks]{hyperref}
\author{Akash Singha Roy}
\address{Department of Mathematics \\ University of Georgia \\ Athens, GA 30602}
\email{akash01s.roy@gmail.com}

\subjclass[2020]{Primary 11A25; Secondary 11N36, 11N37, 11N64, 11N69}

\renewcommand\phi\varphi
\usepackage{accents}

\renewcommand{\pod}[1]{\allowbreak\mathchoice
  {\if@display \mkern 18mu\else \mkern 8mu\fi (#1)}
  {\if@display \mkern 18mu\else \mkern 8mu\fi (#1)}
  {\mkern4mu(#1)}
  {\mkern4mu(#1)}
}
\usepackage{graphicx}
\DeclareMathAlphabet{\curly}{U}{rsfs}{m}{n}

\newcommand{\1}{\mathbbm{1}}

\newcommand{\F}{\mathbb{F}}
\newcommand\Z{\mathbb{Z}}
\allowdisplaybreaks

\newcommand\NatNos{\mathbb N}

\renewcommand\v{\mathbf{v}}

\newtheorem{thm}{Theorem}[section]

\newtheorem{prop}[thm]{Proposition}
\newtheorem{lem}[thm]{Lemma}

\theoremstyle{remark}
\newtheorem*{rmk}{Remark}

\newcommand\err{\mathcal E}
\newcommand\reals{\mathbb R}
\newcommand\complex{\mathbb C}
\newcommand\Udisk{\mathbb U}
\newcommand\sumxyzf{\sum\nolimits_{n \le x}^* f(n)}
\newcommand\expfpp{\exp\left(\sum_{p \le y} \frac{|f(p)|}p \right)}
\newcommand\Ree{\mathrm{Re}}
\newcommand\summ{\sum_{\substack{m \le x\\P^+(m) \le y}}}
\newcommand\sumj{\sum_{j \ge 1}}
\newcommand\sumjT{\sum_{j \ge 2}}
\newcommand\sumPTwoj{\sum_{\substack{P_2, \dots, P_j \in (y, P_1)\\P_2, \dots, P_j \text{ distinct}\\P_2 \cdots P_j \le x/mP_1}}}
\newcommand\sumPTwojMO{\sum_{\substack{P_2, \dots, P_{j-1} \in (y, P_1)\\P_2, \dots, P_{j-1} \text{ distinct}\\P_2 \cdots P_{j-1} \le x/myP_1}}}
\newcommand\summfm{\sum_{\substack{m \le x/z\\P^+(m) \le y}} f(m)}
\newcommand\bbm{\mathbbm 1}
\newcommand\alphatil{\widetilde{\alpha}}
\newcommand\alphatilq{\widetilde{\alpha}(q)}
\newcommand\chibara{\overline\chi(a)}
\newcommand\chisign{\chi(\sigma(n))}

\newcommand\phiqrec{\frac1{\phi(q)}}
\newcommand\phiellrec{\frac1{\phi(\ell^e)}}
\newcommand\elleell{\ell^{e_\ell}}
\newcommand\phielleellrec{\frac1{\phi(\ell^{e_\ell})}}

\newcommand\condofchi{\cond(\chi)}
\newcommand\condofchiell{\cond(\chi_\ell)}
\newcommand\cond{\mathfrak f}
\newcommand\chiZell{\chi_{0, \ell}}
\newcommand\chiZellv{\chiZell(v)}
\newcommand\chiellvvO{\chi_\ell(v^2+v+1)}
\newcommand\sumvmodelle{\sum_{v \bmod \ell^{e_\ell}}}
\newcommand\sumvmodell{\sum_{v \bmod \ell}}
\newcommand\omegaOnefchi{\omega_1(\condofchi)}
\newcommand\omegaTwofchi{\omega_2(\condofchi)}
\newcommand\Modetachialphatil{|\etachi|/\alphatil}
\newcommand\prodellq{\prod_{\ell \mid q}}
\newcommand\Vellew{\mathcal V_{\ell^e} (w)}
\newcommand\SigmatilOne{\widetilde \Sigma_1}
\newcommand\SigmatilTwo{\widetilde \Sigma_2}
\newcommand\Vqtil{\widetilde{\mathcal V}_q}

\newcommand\VellTwotilw{\widetilde{\mathcal V}_{\ell^2}}
\newcommand\NineSixInv{9 \cdot 16^{-1}}
\newcommand\VellTwotilwNineSixInv{\VellTwotilw(\NineSixInv) }
\newcommand\UellTwoCart{ U_{\ell^2} \times U_{\ell^2}}
\newcommand\Velltilw{\widetilde{\mathcal V}_\ell(w)}

\newtheorem{innercustomgeneric}{\customgenericname}
\providecommand{\customgenericname}{}
\newcommand{\newcustomtheorem}[2]{
  \newenvironment{#1}[1]
  {
   \renewcommand\customgenericname{#2}
   \renewcommand\theinnercustomgeneric{##1}
   \innercustomgeneric
  }
  {\endinnercustomgeneric}
}

\newcustomtheorem{customthm}{Theorem}
\newcustomtheorem{customlemma}{Lemma}

\newcommand\sm\setminus

\newcommand\sumvmodq{\sum_{v \bmod q}}

\newcommand\expOlogtwoqOOne{\exp\big(O\big((\log_2 (3q))^{O(1)}\big)\big)}

\newcommand\Fellbar{\overline{\F}_\ell}

\newcommand\mcS{\mathcal S}

\newcommand\etachi{\eta_\chi}
\newcommand\etachiell{\eta_{\chi, \ell}}

\title[Mean Values of Multiplicative Functions and distribution of $\sigma(n)$]{Mean values of multiplicative functions and applications to the distribution of the sum of divisors}
\date{}

\begin{document}

\begin{abstract}  We provide uniform bounds on mean values of multiplicative functions under very general hypotheses, detecting certain power savings missed in known results in the literature. As an application, we study the distribution of the sum-of-divisors function $\sigma(n)$ in coprime residue classes to moduli $q \le (\log x)^K$, 
obtaining extensions of results of \'Sliwa that are uniform in a wide range of $q$ and optimal in various parameters. As a consequence of our results, we obtain that the values of $\sigma(n)$ sampled over $n \le x$ with $\sigma(n)$ coprime to $q$ are asymptotically equidistributed among the coprime residue classes mod $q$, uniformly for odd $q \le (\log x)^K$. On the other hand, if $q$ is even, then equidistribution is restored provided we restrict to inputs $n$ having sufficiently many prime divisors exceeding $q$.
\end{abstract}

\keywords{multiplicative function, mean values, equidistribution, uniform distribution, weak equidistribution, weak uniform distribution, sum of divisors}
\maketitle
\section{Introduction}
Mean values of multiplicative functions have been a central topic of interest in analytic number theory. There is already a rich variety of literature on this subject, with the classical result of Hal\'asz \cite{halasz68} (see \cite[Corollary III.4.12]{tenenbaum15} for a precise version) providing a widely applicable general bound on the partial sums of any multiplicative function taking values in the (complex) unit disk, and with the analytic method of Landau--Selberg--Delange (see \cite[Chapter II.5]{tenenbaum15}) providing a highly precise asymptotic formula for the partial sums of multiplicative functions whose Dirichlet series behave like powers of the Riemann zeta function. However in applications, situations arise where Hal\'asz's bound turns out to not be precise enough, while the Landau--Selberg--Delange method despite being applicable leads to additional error terms that become too large to yield useful information. Such situations occur, for instance, while studying the distribution of values of multiplicative functions in residue classes. 

In an attempt to address this issue, in recent joint work with Pollack \cite{PSR23A}, we gave an upper bound on the partial sum of a multiplicative function taking values in the unit disk, under some natural control on the average behavior of this function at the primes. Using this result, we studied the distribution of the values of the Alladi-Erd\H{o}s function $A(n) \coloneqq \sum_{p^k \parallel n} k \cdot p$ and Euler's totient function $\phi(n) \coloneqq \#\{1 \le d \le n: ~ \gcd(d, n)=1 \}$ in residue classes to moduli varying uniformly in a wide range. Our results extended those of Goldfeld \cite{goldfeld17} and Narkiewicz \cite{narkiewicz66} (who studied the corresponding questions for fixed moduli) and had error terms that (for concrete heuristic reasons) could be expected to be essentially optimal. A related function occurring commonly in number theory is the sum-of-divisors function $\sigma(n) \coloneqq \sum_{d \mid n} d$, and \'Sliwa \cite{Sli73} studies the distribution of the values of this function among the coprime residue classes to a fixed modulus $q$, obtaining a necessary and sufficient criterion for $\sigma(n)$ to be equidistributed among the coprime residue classes mod $q$. Our methods in \cite{PSR23A}, however, are unable to obtain a complete uniform version of \'Sliwa's results, primarily because they are unable to detect the necessary power savings in certain naturally-arising character sums. Motivated by this problem, our first main result in this manuscript is one which provides a variant of Theorem 1.1 in \cite{PSR23A} that can detect certain power savings missed in the referred result, and in fact also enables us to obtain a uniform version of \'Sliwa's result.     

In what follows, $\pi(y)$ denotes (as usual) the number of primes up to $y$, and $\Udisk$ denotes the unit disk in the complex plane, namely, the set $\{s \in \complex: |s| \le 1\}$. For a given real number $z$, we say that a positive integer $n$ is $z$\textsf{-smooth} if it has no prime factor exceeding $z$. The number of $z$-smooth numbers up to $x$ is denoted by $\Psi(x, z)$, a quantity which has been studied extensively in the literature. 
\begin{thm}\label{thm:mainbound}
Let $f \colon \NatNos \rightarrow \Udisk$ be a multiplicative function and $x, y, z, M$ be positive real numbers such that $M \ge 1$, $1<z<x$ and $e^{11/2} \le y \le z^{1/(18\log \log z)^2}$. Assume that there exists a complex number $\varrho \in \mathbb U$ such that for all $Y \ge y$,
\begin{equation}\label{eq:mainhypo}
\sum_{y<p \le Y} f(p) = \varrho (\pi(Y) - \pi(y)) + O(M Y \err(y)), \end{equation}
for some decreasing function $\err \colon \reals^+ \rightarrow \reals^+$ satisfying $\lim_{X \rightarrow \infty} \err(X) = 0$. Then
\begin{multline*}
\sum_{n \le x} f(n) \ll \frac{|\varrho| x}{\log z} \left(\frac{\log z}{\log y}\right)^{\Ree(\varrho)} \expfpp \\+ ~ \frac{x (\log y)^{1+|\varrho|}}{(\log z)^{2-\Ree(\varrho)}} \expfpp + \Psi(x, z) + Mx (\log x)^2 \left(\err(y)+\frac1y\right),
\end{multline*}
where the implied constant depends at most on the implied constant in \eqref{eq:mainhypo}. 
\end{thm}
A few comments are in order. In applications, the terms involving $\Psi(x, z)$, $\err(y)$ and $1/y$ often become very small, so the usefulness of the bounds is dictated by the sizes of the first two terms. The crucial difference between this bound and the one in Theorem 1.1 of \cite{PSR23A} is the factor $(\log z)^{\Ree(\varrho)}$ in the first expression, whose analogue in \cite[Theorem 1.1]{PSR23A} contained the larger factor $(\log x)^{|\varrho|}$. These greater savings turn out to play a significant role in certain applications. However, our bound above cannot entirely subsume \cite[Theorem 1.1]{PSR23A} because the additional term with $(\log y)^{1+|\varrho|}$ (which has no analogue in \cite[Theorem 1.1]{PSR23A}) can limit the applicability of our bound. We shall witness these phenomena in our applications (see the remark following the proof of Proposition \ref{prop:SigmaEvenConvenient}).

Both Theorem \ref{thm:mainbound} and the arguments used to prove it can be applied to study the asymptotic distribution of multiplicative functions like the sum-of-divisors function $\sigma(n)$ in coprime residue classes to moduli varying uniformly in a wide range. Following Narkiewicz  \cite{narkiewicz66}, we say that an integer-valued arithmetic function $g$ is \textsf{weakly uniformly distributed} (or \textsf{weakly equidistributed}) modulo $q$ if $\gcd(g(n),q)=1$ for infinitely many $n$ and, for every \emph{coprime} residue class $a\bmod{q}$,
\begin{equation}\label{eq:WUDrelation} \#\{n\le x: g(n)\equiv a\pmod{q}\} \sim \frac{1}{\phi(q)} {\#\{n\le x: \gcd(g(n),q)=1\}} , \quad\text{as $x\to\infty$}. \end{equation}
In \cite[Theorem I]{narkiewicz66}, Narkiewicz gives a criterion for deciding weak equidistribution in a large class of multiplicative functions. As an application, we can show that $\sigma(n)$ is weakly equidistributed mod $q$ for any $q$ coprime to $6$. This result was extended by \'Sliwa \cite{Sli73} (see also \cite[Proposition 7.9, p. 106]{NarBook}), who showed that $\sigma(n)$ is weakly equidistributed to a fixed modulus $q$ precisely when $q$ is not divisible by $6$. 

As a natural analogue of the Siegel-Walfisz theorem for primes in arithmetic progressions, it seems of some interest to investigate analogues of these results where the modulus $q$ itself is allowed to vary uniformly in a certain range depending on the stopping point $x$ of inputs. To formalize this, we say that an integer-valued arithmetic function $g(n)$ is \textsf{weakly equidistributed} mod $q$, \textsf{uniformly for $q \le (\log x)^{K}$,} if:
\begin{enumerate}
\item[(i)] For every such $q$, $g(n)$ is coprime to $q$ for infinitely many $n$, and
\item[(ii)] The relation \eqref{eq:WUDrelation} holds uniformly in moduli $q \le (\log x)^K$ and in coprime residue classes $a$ mod $q$. Explicitly, this means that for any $\epsilon>0$, there exists $X(\epsilon)>0$ such that the ratio of the left hand side of \eqref{eq:WUDrelation} to the right hand side lies in $(1-\epsilon, 1+\epsilon)$ for all $x>X(\epsilon)$, $q \le (\log x)^K$ and coprime residues $a$ mod $q$.   
\end{enumerate}
The weak equidistribution of certain classes of arithmetic functions to uniformly varying moduli was investigated in \cite{LPR21, PSR22, PSR23}. A general theorem established in \cite{PSR23} yields the weak equidistribution of $\sigma(n)$ uniformly to moduli $q \le (\log x)^K$ coprime to $6$. However, the methods in \cite{PSR23} are unable to address the case when $q$ is even but not divisible by $6$ (a case in which \'Sliwa is able to show that $\sigma(n)$ is weakly equdistributed mod $q$). Moreover, the error terms that arise from the arguments in \cite{PSR23} are quite weak. Our next three theorems address these defects. We first establish an effective estimate demonstrating the weak equidistribution of $\sigma(n)$ to odd moduli $q \le (\log x)^K$ and obtain a strong error term that can be expected to be (essentially) of the correct order of magnitude. 
\begin{thm}\label{thm:sigmaodd}
Fix $K>0$ and $\epsilon \in (0, 1)$.  We have
\begin{multline*}
\#\{n \le x: \sigma(n) \equiv a \pmod q\}\\ = \phiqrec \#\{n \le x: \gcd(\sigma(n), q)=1\} + O_{K, \epsilon}\left(\frac x{\phi(q)(\log x)^{1-\alpha(q)(1/3+\epsilon)}}\right),
\end{multline*}
uniformly in odd moduli $q \le (\log x)^K$ and coprime residue classes $a$ mod $q$, where $\alpha(q) \coloneqq \prod_{\ell \mid q} \left(1-\frac1{\ell-1}\right)$. As a consequence, $\sigma(n)$ is weakly equidistributed mod $q$, uniformly for odd moduli $q \le (\log x)^K$.
\end{thm}
To obtain uniform analogues of \'Sliwa's results, we still need to address the case when the modulus $q$ is even. Here, the first crucial difference with the above theorem is that while studying the weak equidistribution of $\sigma(n)$ to even moduli, the set of relevant inputs $n \le x$ is highly sparse. In fact, elementary number theory shows that if $\sigma(n)$ is odd, then $n$ must be of the form $2^k m^2$ for some odd $m$, which means that there are $O(x^{1/2})$ many $n \le x$ which have $\sigma(n)$ coprime to a given even modulus $q$. Sparse sets like this can often present difficulties while studying arithmetic questions about them. However, we can work around this sparsity issue by looking at the behaviour of $\sigma(n)$ at the squares of primes. 

Another crucial difference between the situations for odd and even moduli is that in the latter case, the positive integers $n$ with too few large prime factors present an obstruction to uniformity in the modulus. As such, in order to restore uniformity in the modulus $q$, it becomes necessary to restrict the set of inputs $n$ to those having sufficiently many prime factors exceeding $q$. To make this precise, we write $P(n)$ or $P^+(n)$ for the largest prime divisor of $n$, with the convention that $P(1)=1$. We set $P_1(n) \coloneqq P(n)$ and define, inductively, $P_k(n) \coloneqq P_{k-1}(n/P(n))$. Thus, $P_k(n)$ is the $k$th largest prime factor of $n$ (counted with multiplicity), with $P_k(n)=1$ if $\Omega(n) < k$. 

In what follows, we set 
$$\alphatil(q) \coloneqq \frac1{\phi(q)} \#\{v \bmod q: \gcd(v(v^2+v+1), q)=1\} = \prod_{\substack{\ell \mid q\\ \ell \equiv 1 \pmod 3}} \left(1-\frac2{\ell-1}\right),$$
the last equality being true by the Chinese Remainder Theorem and the law of quadratic reciprocity.  
\begin{thm}\label{thm:sigmaeven}
Fix $K>0$ and $\epsilon \in (0, 1)$.  We have
\begin{multline}\label{eq:sigmaeven}
\#\{n \le x: P_6(n)>q, ~ \sigma(n) \equiv a \pmod q\}\\ = \phiqrec \#\{n \le x: P_6(n)>q, ~ \gcd(\sigma(n), q)=1\} + O_{K, \epsilon}\left(\frac{x^{1/2}}{\phi(q)(\log x)^{1-\alphatil(q)(1/4+\epsilon)}}\right),
\end{multline}
uniformly in even moduli $q \le (\log x)^K$ not divisible by $3$, and in coprime residues $a$ mod $q$. 
\end{thm}
In subsection \ref{subsec:P6Optimal}, we give an explicit counterexample showing that the restriction $P_6(n)>q$ is optimal in the sense that it cannot be replaced by a condition of the form ``$P_k(n)>q$" with $k<6$, while still retainining weak equidistribution among the corresponding set of $n$'s in the same range of uniformity in $q$. In fact, we will show that uniformity fails to moduli $q$ of the form $2Q^2$ for certain odd squarefree integers $Q$ having several prime factors. While the proof of Theorem \ref{thm:sigmaeven} relies heavily on bounds on certain character sums modulo prime powers, the aforementioned counterexample works due to an excess in the number of lifts of solutions to a polynomial congruence from prime moduli to prime square moduli. 

On the other hand, if we restrict our moduli $q$ to be \textit{squarefree} even integers, then we may enlarge the set of inputs $n$ to those having four (as opposed to six) prime divisors exceeding $q$. 
\begin{thm}\label{thm:sigmaevensqfree}
Fix $K>0$ and $\epsilon \in (0, 1)$.  We have
\begin{multline}\label{eq:sigmaevensqfree}
\#\{n \le x: P_4(n)>q, ~ \sigma(n) \equiv a \pmod q\}\\ = \phiqrec \#\{n \le x: P_4(n)>q, ~ \gcd(\sigma(n), q)=1\} + O_{K, \epsilon}\left(\frac{x^{1/2}}{\phi(q)(\log x)^{1-\alphatil(q)(1/4+\epsilon)}}\right),
\end{multline}
uniformly in squarefree even moduli $q \le (\log x)^K$ not divisible by $3$, and in coprime residue classes $a$ mod $q$.
\end{thm}
In subsection \ref{subsec:P4Optimal}, we show that the restriction $P_4(n)>q$ is optimal for squarefree even $q$ in the same sense as mentioned above. The constants ``$1/3$" (respectively, ``$1/4$") in the exponent of $\log x$ in the error terms of Theorem \ref{thm:sigmaodd} (resp., Theorems \ref{thm:sigmaeven} and \ref{thm:sigmaevensqfree}), arise as the maximum real part of the averages 
$$\rho_\chi \coloneqq \phiqrec \sum_{\substack{v \bmod q\\\gcd(v, q)=1}} ~ \chi(v+1) ~ ~ ~ \text{  \Bigg(resp. } \eta_\chi \coloneqq \phiqrec \sum_{\substack{v \bmod q\\\gcd(v, q)=1}} ~ \chi(v^2+v+1)\text{\Bigg)}$$
taken over the nonprincipal Dirichlet characters $\chi$ mod $q$. These averages play the roles of the parameter $\varrho$ in the corresponding analogues of Theorem \ref{thm:mainbound}, when the role of the multiplicative function $f$ is played by the functions $\chi \circ \sigma$ that arise by applying the orthogonality of Dirichlet characters to detect the congruence $\sigma(n) \equiv a \pmod q$. As we shall see in the course of our arguments, the maximum values of $\Ree(\rho_\chi)$ (resp. $\Ree(\eta_\chi)$) can be attained by characters $\chi$ of conductor $15$ (resp. conductors $5$, $7$, $13$, $35$), in the case when $q$ is divisible by the respective conductors (see the remarks at the end of the proofs of Theorem \ref{thm:sigmaodd} and Proposition \ref{prop:etachiBounds}). This suggests (but does not prove) that the constants ``$1/3$" and ``$1/4$" in the error terms of Theorems \ref{thm:sigmaodd}, \ref{thm:sigmaeven} and \ref{thm:sigmaevensqfree} cannot be replaced by smaller constants in general.

We conclude this introductory section with the remark that precise estimates for the main terms in Theorems \ref{thm:sigmaodd}, \ref{thm:sigmaeven} and \ref{thm:sigmaevensqfree} can be obtained from Theorems A and B in work of Scourfield \cite{scourfield85} (see \cite{scourfield84} and \cite[Proposition 2.1]{PSR22} for similar results). Finally, in the course of our arguments, we shall obtain sharp upper bounds for the character sums $\chi(\sigma(n))$ for nontrivial Dirichlet characters $\chi$ to our moduli $q$. In the range $q \le (\log x)^K$, these can be thought of as improved analogues of work of Balasuriya, Shparlinski and Sutantyo \cite{BSS09}. Related work on exponential sums involving $\sigma(n)$ has been (implicitly) done in \cite{BS04} and \cite{BS06}, while \cite{BL07} and \cite{BBS08} investigate the corresponding sums for the aliquot-divisor-sum function $s(n) \coloneqq \sigma(n) - n = \sum_{d \mid n\colon d \ne n} d$.
\subsection*{Notation and conventions:} We do not consider the zero function as multiplicative (thus, if $f$ is multiplicative, then $f(1)=1$). As mentioned above, we shall use $P^+(n)$ or $P(n)$ to denote the largest prime divisor of $n$, and $P_k(n)$ for the $k$-th largest prime factor of $n$. In addition, we denote the least prime divisor of $n$ by $P^-(n)$. When there is no danger of confusion, we write $(a,b)$ instead of $\gcd(a,b)$. Throughout, the letters $p$ and $\ell$ are to be understood as denoting primes. Implied constants in $\ll$ and $O$-notation may always depend on any parameters declared as ``fixed''; other dependence will be noted explicitly (for example, with subscripts). We write $\log_{k}$ for the $k$th iterate of the natural logarithm. We shall use $U_q$ to denote the group of units (or the multiplicative group) modulo a positive integer $q$.
\section{Uniform bounds on the partial sums of multiplicative functions: Proof of Theorem \ref{thm:mainbound}}
Our arguments for Theorems \ref{thm:mainbound} and \ref{thm:sigmaodd} begin similarly to those given for \cite[Theorems 1.1 and 1.3]{PSR23A}, but we include the complete argument here in order to keep the exposition self-contained. We first bound the contribution of the $n \le x$ that are either $z$-smooth or have a repeated prime factor exceeding $y$. Since $|f(n)| \le 1$, the contribution of the former $n$ is bounded in absolute value by $\Psi(x, z)$, while that of the latter $n$ is bounded in absolute value by
$$\sum_{p>y} \sum_{\substack{n \le x\\p^2 \mid n}} 1 \le x\sum_{p>y} \frac1{p^2} \ll \frac xy.$$
Both of these are absorbed in the expressions given in the claimed bounds. 

Let $\sumxyzf$ denote the sum over the remaining $n$, namely those that have $P^+(n)>z$ and no repeated prime factor exceeding $y$. Any $n$ counted in this sum can be uniquely written in the form $m P_j \cdots P_1$ for some $j \ge 1$, where $P_1 = P^+(n) > z$ and $P^+(m) \le y < P_j < \dots < P_1$. As such $f(n) = f(m) f(P_j) \cdots f(P_1)$ and 
\begin{align*}\allowdisplaybreaks
\sumxyzf &= \sum_{j \ge 1} \summ f(m) \sum_{\substack{P_1, \dots, P_j\\ P_1>z,~ P_j \cdots P_1 \le x/m\\y < P_j < \dots < P_1}} f(P_j) \cdots f(P_1)\\
&= \sumj \summ f(m) \sum_{\substack{P_2, \dots, P_j\\\ P_j \cdots P_2 \le x/mz\\y<P_j < \dots < P_2\\}} f(P_j) \cdots f(P_2) \sum_{\max\{P_2, z\}<P_1 \le x/m P_2 \cdots P_j} f(P_1).
\end{align*}
Here $\max\{P_2, z\}$ is to be replaced by $z$ in the case $j=1$.

Invoking hypothesis \eqref{eq:mainhypo} to estimate the innermost sum on $P_1$, we find that
\begin{equation}\label{eq:PostP1Removal}
\begin{split}
\sumxyzf = \varrho\sumj \frac1{(j-1)!} \summ f(m) \sum_{z<P_1 \le x/my^{j-1}} \sumPTwoj &f(P_2) \cdots f(P_j)\\ &+ O(Mx\err(y) \log x),
\end{split}   
\end{equation}
where we have observed that the total size of the resulting error term incurred upon an application of \eqref{eq:mainhypo} is
$$\ll Mx \err(y) \sumj \sum_{\substack{m, P_2, \dots, P_j\\mP_2 \cdots P_j \le x/z\\P^+(m) \le y < P_j < \dots < P_2}} \frac1{mP_2 \cdots P_j} \le Mx \err(y) \sum_{n \le x/z} \frac1n \ll Mx\err(y) \log x.$$
Now for $j \ge 2$ and each $i \in \{2, \dots, j\}$, estimate \eqref{eq:mainhypo} shows that
\begin{equation}\label{eq:SumfPi}
\begin{split}
\sum_{\substack{y<P_i \le x/mP_1 \cdots P_{i-1}P_{i+1} \cdots P_j \\ P_i<P_1, ~ r \ne i \implies P_r \ne P_i}} f(P_i) &= \sum_{\substack{y<P_i \le x/mP_1 \cdots P_{i-1}P_{i+1} \cdots P_j\\P_i < P_1}} f(P_i) + O(j)\\
&= \varrho\sum_{\substack{y<P_i \le x/mP_1 \cdots P_{i-1}P_{i+1} \cdots P_j\\P_i<P_1, ~ r \ne i \implies P_r \ne P_i}} 1 + O\left(j + \frac{Mx}{mP_1 \cdots P_{i-1}P_{i+1} \cdots P_j} \err(y)\right).
\end{split}
\end{equation}
For each $j \ge 2$, we use this estimate for $i \in \{2, \dots, j\}$ in order to successively remove the $f(P_2), \dots, f(P_j)$ occurring in the main term of \eqref{eq:PostP1Removal}. To this end, we define
$$\widetilde{\err} \coloneqq j \sumPTwojMO 1 + \frac{Mx \err(y)}{mP_1} \sumPTwojMO \frac1{P_2 \cdots P_{j-1}},$$
and write
\begin{align*}
\sumPTwoj f(P_2) \cdots f(P_j) &= \sum_{\substack{P_3, \dots, P_j \in (y, P_1)\\P_3, \dots, P_j \text{ distinct}\\P_3 \cdots P_j \le x/myP_1}} f(P_3) \cdots f(P_j) \sum_{\substack{y<P_2 \le x/mP_1 P_3 \cdots P_j \\ P_2<P_1, ~ r \ne 2 \implies P_2 \ne P_r}} f(P_2)\\
&= \varrho \sum_{\substack{P_2, \dots, P_j \in (y, P_1)\\P_2, \dots, P_j \text{ distinct}\\P_2 \cdots P_j \le x/mP_1}} f(P_3) \cdots f(P_j) + O(\widetilde{\err}),
\end{align*}
where in the last step we have noted that the error term resulting from the application of \eqref{eq:SumfPi} for $i=2$ is, by relabelling, equal to $\widetilde{\err}$. Likewise, invoking \eqref{eq:SumfPi} for $i = 3, \dots, j$, we obtain 
$$\sumPTwoj f(P_2) \cdots f(P_j) = \varrho^{j-1} \sumPTwoj 1 + O(j \widetilde{\err}).$$
Inserting this into \eqref{eq:PostP1Removal} for each $j \ge 2$ incurs a total error of size 
\begin{align*}\allowdisplaybreaks
&\ll \sum_{j \ge 2} \frac1{(j-1)!} \summ \sum_{z<P_1 \le \frac{x}{my^{j-1}}} \Bigg\{j^2 \sumPTwojMO 1 + \frac{Mxj \err(y)}{mP_1} \sumPTwojMO \frac1{P_2 \cdots P_{j-1}}\Bigg\}\\
&\ll (\log x)^2 \sum_{j \ge 2} \sum_{\substack{m, P_1, \dots, P_{j-1}\\mP_1 \cdots P_{j-1} \le x/y\\P_1>z, ~ P^+(m) \le y< P_{j-1} < \dots <P_1}} 1 ~ + ~ Mx (\log x) \err(y) \sum_{\substack{m, P_1, \dots, P_{j-1}\\mP_1 \cdots P_{j-1} \le x/y\\P_1>z, ~ P^+(m) \le y< P_{j-1} < \dots <P_1}} \frac1{m P_1 \cdots P_{j-1}}\\
&\ll (\log x)^2 \sum_{n \le x/y} 1 ~ + ~ M x (\log x) \err(y) \sum_{n \le x/y} \frac1n \ll Mx (\log x)^2 \left(\err(y) + \frac1y\right).
\end{align*}
As a consequence, we obtain
\begin{equation}\label{eq:PostP1PjRemoval}
\begin{split}
\sumxyzf &= \varrho\sumj \frac{\varrho^{j-1}}{(j-1)!} \summ f(m) \sum_{\substack{P_1, \dots, P_j\\P_1>z, ~ P_1 \cdots P_j \le x/m\\P_2 \cdots P_j \in (y, P_1) \text{ distinct}}} 1 + O\left(Mx (\log x)^2 \left(\err(y) + \frac1y\right)\right)\\
&= \summfm \sumj \sum_{\substack{P_1, \dots, P_j\\P_1>z, ~ P_1 \cdots P_j \le x/m\\y < P_j < \dots < P_1}} \varrho^j ~ + ~ O\left(Mx (\log x)^2 \left(\err(y) + \frac1y\right)\right).
\end{split}
\end{equation}

At this point, we observe that every squarefree positive integer $r \le x/m$ having $P^+(r) > z$ and $P^-(r)>y$ can be uniquely written in the form $P_j \cdots P_1$ with $P_1>z$ and $y<P_j < \dots < P_1$, and in this case $j = \Omega(r)$. As such, the main term in the final expression in \eqref{eq:PostP1PjRemoval} is equal to 
\begin{equation}\label{eq:MainTerm2}
\summfm \sum_{\substack{r \le x/m\\r\text{ squarefree}\\P^+(r)>z, ~ P^-(r)>y}} \varrho^{\Omega(r)} ~.     
\end{equation}
Ignoring the condition $P^+(r)>z$ incurs a total error of size
$$\ll \sum_{\substack{m \le x/z\\ P^+(m) \le y}} ~ ~ \sum_{\substack{r\le x/m\\ P^-(r)>y, \, P^+(r) \le z}} 1 \le \sum_{\substack{n \le x\\P^+(n) \le z}} 1 \le \Psi(x, z),$$
which is absorbed in the bound claimed in the theorem statement. Moreover, since any non-squarefree $r$ having $P^-(r)>y$ is divisible by the square of a prime exceeding $y$, ignoring the squarefreeness condition in \eqref{eq:MainTerm2} incurs a total error 
$$\ll \sum_{\substack{m \le x/z\\ P^+(m) \le y}} ~ \sum_{p>y} ~ \sum_{\substack{r\le x/m\\ p^2 \mid r}} 1 \ll x\sum_{\substack{m \le x/z\\ P^+(m) \le y}} \frac1m \sum_{p>y} \frac1{p^2} \ll \frac x{y \log y} \prod_{\ell \le y} \left(1+\sum_{v \ge 1} \frac1{\ell^v}\right) \ll \frac xy,$$
which is also absorbed in the claimed bound. Hence, up to a negligible error, the expression in \eqref{eq:MainTerm2} is equal to 
\begin{equation}\label{eq:MainTerm3}
\summfm \sum_{r \le x/m} \bbm_{P^-(r)>y} ~ \varrho^{\Omega(r)}.    
\end{equation}
In order to estimate the innermost sum, we shall be making use of the lemma below, which we will establish in the next section. 
\begin{lem}\label{lem:roughtwistedLiouville}
Let $X, Y, Z \ge e^{11/2}$ be positive real numbers satisfying $Y \le Z^{1/(18\log \log Z)^2}$ and $Z \le X$. We have for all $\beta \in \Udisk$,
$$\sum_{n \le X} \bbm_{P^-(n)>Y} ~ \beta^{\Omega(n)} = \frac X{(\log X)^{1-\beta}} \left\{\frac{e^{-\gamma \beta}}{\Gamma(\beta) (\log Y)^\beta}\left(1+O(\exp(-C_0\sqrt{\log y}))\right) + O\left(\frac{(\log Y)^{1+|\beta|}}{\log Z}\right)\right\},$$
where $C_0>0$ is an absolute constant. In particular,
$$\sum_{n \le X} \bbm_{P^-(n)>Y} ~ \beta^{\Omega(n)} \ll \frac X{(\log X)^{1-\Ree(\beta)}} \left\{\frac{|\beta|}{(\log Y)^{\Ree(\beta)}} + \frac{(\log Y)^{1+|\beta|}}{\log Z}\right\}.$$
The implied constants in the above estimates are absolute.
\end{lem}
Applying the above lemma with $X \coloneqq x/m$, $Y \coloneqq y$, $Z \coloneqq z$ and $\beta \coloneqq \varrho$, we find that the expression in \eqref{eq:MainTerm3} is 
\begin{align*}
&\ll \sum_{\substack{m \le x/z\\P^+(m) \le y}} |f(m)| \cdot \frac{x/m}{(\log (x/m))^{1-\Ree(\varrho)}} \left\{\frac{|\varrho|}{(\log y)^{\Ree(\varrho)}} + \frac{(\log y)^{1+|\varrho|}}{\log z}\right\}\\
&\ll \frac x{(\log z)^{1-\Ree(\varrho)}} \left(\sum_{m: \, P^+(m) \le y} \frac{|f(m)|}m \right) \left\{\frac{|\varrho|}{(\log y)^{\Ree(\varrho)}} + \frac{(\log y)^{1+|\varrho|}}{\log z}\right\}.
\end{align*}
Finally, since $|f(n)| \le 1$ for all $n$, the sum on $m$ is $\ll \exp(\sum_{p \le y} |f(p)|/p)$, yielding the estimate in the theorem. This completes the proof of Theorem \ref{thm:mainbound}, up to the proof of Lemma \ref{lem:roughtwistedLiouville}.
\section{Proof of Lemma \ref{lem:roughtwistedLiouville}: the Landau--Selberg--Delange method}
A comprehensive account of the method of Landau--Selberg--Delange may be found in Tenenbaum \cite[Chapter II.5]{tenenbaum15}. However, we shall be using a recent formulation of this method due to Chang and Martin \cite{CM20}. This is based on Tenenbaum's treatment but is more explicit in the dependence on certain parameters, a feature that shall be crucial in our current application. 
\subsection{Setup} In what follows, we write complex numbers $s$ as $s=\sigma+i t$, where $\sigma \coloneqq \Ree(s)$ and $t \coloneqq \mathrm{Im}(s)$. (This convention is relevant only for this section and the use of $\sigma$ will not create any confusion with the notation for the sum-of-divisors function.) For a non-negative $u$, we use $\log^{+}{u}$ to denote the quantity $\max\{0,\log{u}\}$ (the positive part of $\log u$), with the convention that $\log^{+}{0}=0$. 

Given $\delta \in (0, 1]$, a complex number $z$, and positive real numbers $c_0$ and $M$, we say that the Dirichlet series $F(s)$ has \textsf{property $\mathcal P(z; c_0, \delta, M)$} if the function
$$G(s; z):= F(s)\zeta(s)^{-z}$$
satisfies the following two conditions:
\begin{enumerate}
\item[(i)] $G(s; z)$ continues analytically into the region $\sigma \ge 1-c_0/(1+\log^{+}|t|)$, and
\item[(ii)] $|G(s; z)| \leq M(1+|t|)^{1-\delta}$
for all $s$ in this same region. 
\end{enumerate} 

Given complex numbers $z$ and $w$, along with positive numbers $c_0$, $M$ and $\delta \in (0, 1]$, we say that a Dirichlet series $F(s) := \sum_{n=1}^\infty a_n n^{-s}$ has \textsf{type $\mathcal T(z, w; c_0, \delta, M)$} if:
\begin{enumerate}
\item[(i)] $F(s)$ has property $\mathcal P(z; c_0, \delta, M)$, and 
\item[(ii)] there is a sequence $\{b_n\}_{n=1}^\infty$ of nonnegative real numbers satisfying $|a_n| \leq b_n$ for all $n$, such that the Dirichlet series $\sum_{n=1}^\infty b_nn^{-s}$ has property $\mathcal P(w; c_0, \delta, M)$. 
\end{enumerate}

The following is the special case  of Theorem A.13 in \cite{CM20} with $N = 0$.
\begin{prop}\label{prop:CM} Fix $A \ge 1$ and let $z, w$ be complex numbers satisfying $|z|, |w|\le A$. Let $c_0, \delta, M$ be positive real numbers with $c_0 \le 2/11$ and $\delta \le 1$. Let $F(s)= \sum_{n=1}^{\infty} a_n/n^s$ be a Dirichlet series of type $\mathcal T(z, w; c_0, \delta, M)$. Then, uniformly for $x\ge \exp(8^{1/A} \max\{\delta^{-1-1/A}, 2/c_0\})$,
we have 
\[ \sum_{n \le x} a_n = \frac x{(\log{x})^{1-z}}\left(\frac{G(1;z)}{\Gamma(z)} + O(M R(x))\right),\]
where
\[ R(x) = \left(\frac1{\delta^{2A+3/2}} + \frac1{c_0^{2A+1}}\right) \exp\left(-\frac{1}{6}\sqrt{c_0 \delta \log{x}}\right) + \frac{1}{c_0\log{x}}\]
and the implied constant depends at most on $A$.
\end{prop}
Here we have corrected some typos in \cite{CM20}; the expression for $R(x)$ there has an extra factor of $M$ throughout as well as an extra factor of $x$ in its first term.

\subsection{Proof of Lemma \ref{lem:roughtwistedLiouville}.} We claim that the Dirichlet series $F(s) \coloneqq \sum_{n \ge 1} \bbm_{P^-(n)>Y} ~ \beta^{\Omega(n)}/n^s$ is of type $\mathcal T(\beta, |\beta|; 1/\log Y, 1, C_0 (\log Y)^{|\beta|})$ for some absolute constant $C_0>0$. Indeed, in the half plane $\sigma>1$, we find that 
\begin{equation}\label{eq:G(s)EuelerProd}
G(s) \coloneqq F(s) \zeta(s)^{-\beta} = \prod_{p \le Y} \left(1-\frac1{p^s}\right)^{\beta} \prod_{p > Y} \left(1-\frac1{p^s}\right)^{\beta} \left(1-\frac{\beta}{p^s}\right)^{-1},
\end{equation}
and in the same half plane
\begin{equation}\label{eq:logG(s)}
\log G(s) = \beta\sum_{p \le Y} \log\left(1-\frac1{p^s}\right) + \sum_{p>Y} \left\{\beta\log\left(1-\frac1{p^s}\right) - \log\left(1-\frac{\beta}{p^s}\right)\right\}.
\end{equation}
Now since $Y \ge e^{11/2}$, we see that if $\sigma \ge 1-1/\log Y$, then $\sigma \ge 9/11$, so that $|\beta/p^s| \le 2^{-\sigma} \le 0.57$. Consequently for all such $s$, the second sum in \eqref{eq:logG(s)} is 
$$\sum_{p>Y} \left\{\beta\left(-\frac1{p^s} + O\left(\frac1{p^{2\sigma}}\right)\right) + \left(\frac\beta{p^s} + O\left(\frac1{p^{2\sigma}}\right)\right)\right\} \ll \sum_{p>Y} \frac1{p^{2\sigma}} \le \sum_{p>Y} \frac1{p^{18/11}} \ll 1.$$
This shows that the second sum in \eqref{eq:logG(s)} converges absolutely and uniformly in the half plane $\sigma \ge 1-1/\log Y$, thus defining a holomorphic function in the same region. Furthermore, for all $s$ in this half plane, we have 
\begin{align*}
\log|G(s)| \le |\log G(s)| &\le |\beta| \sum_{p \le Y} \frac1{p^\sigma} + O(1) \le |\beta| \sum_{p \le Y} \frac1p \exp\left(\frac{\log p}{\log Y}\right) + O(1)\\
&= |\beta|\sum_{p \le Y} \frac1p + O\left(1+\frac1{\log Y}\sum_{p \le Y} \frac{\log p}p\right) = |\beta| \log_2 Y + O(1). 
\end{align*}
We deduce that $|G(s)| \le C_0 (\log Y)^{|\beta|}$ for all $\sigma \ge 1-1/\log Y$, which shows that $F(s)$ has property $\mathcal P(\beta; 1/\log Y, 1, C_0 (\log Y)^{|\beta|})$. By invoking this very observation with $|\beta|$ in place of $\beta$, we find that the Dirichlet series $\sum_{n \ge 1} |\beta|^{\Omega(n)}/n^s$ has property $\mathcal P(|\beta|; 1/\log Y, 1, C_0 (\log Y)^{|\beta|})$, which establishes our claim. 

An application of Proposition \ref{prop:CM} with $A = 1$ now shows that for all $X \ge Y^{16} = \exp(16 \log Y)$, we have 
\begin{equation*}
\sum_{n \le X} \bbm_{P^-(n)>Y} \beta^{\Omega(n)} = \frac X{(\log X)^{1-\beta}} \left\{\frac{G(1)}{\Gamma(\beta)} + O\left(\frac{(\log Y)^{1+|\beta|}}{\log X} + (\log Y)^{3+|\beta|} \exp\left(-\frac16\sqrt{\frac{\log X}{\log Y}}\right)\right)\right\}.
\end{equation*}
For $Y \le Z^{1/(18\log_2 Z)^2}$, we have $2\log_2 Y + \log_2 Z \le 3\log_2 Z \le \frac16\sqrt{\frac{\log Z}{\log Y}}$, which shows that 
\begin{equation}\label{eq:LSDApp}
\sum_{n \le X} \bbm_{P^-(n)>Y} \beta^{\Omega(n)} = \frac X{(\log X)^{1-\beta}} \left\{\frac{G(1)}{\Gamma(\beta)} + O\left(\frac{(\log Y)^{1+|\beta|}}{\log Z}\right)\right\}.
\end{equation}
Finally, 
$$\sum_{p>Y} \left\{\beta\log\left(1-\frac1p\right) - \log\left(1-\frac\beta p\right)\right\} \ll \sum_{p>Y} \frac1{p^2} \ll \frac1{Y\log Y},$$
and by the prime number theorem (with the usual de la Vall\'ee Poussin error term), we have
$$\prod_{p \le Y} \left(1-\frac1p\right)^\beta = \frac{e^{-\gamma \beta}}{(\log Y)^\beta}\left(1+O(\exp(-C_0\sqrt{\log Y}))\right),$$
for some absolute constant $C_0>0$. This shows that 
$$G(1) = \prod_{p \le Y} \left(1-\frac1p\right)^\beta \prod_{p > Y} \left(1-\frac1p\right)^{\beta} \left(1-\frac\beta p\right)^{-1} = \frac{e^{-\gamma \beta}}{(\log Y)^\beta}\left(1+O(\exp(-C_0\sqrt{\log Y}))\right)$$
which from \eqref{eq:LSDApp} yields the first of the claimed estimates in Lemma \ref{lem:roughtwistedLiouville}. Since the Gamma function has no zeros in the complex plane, we have $|\Gamma(s)| \gg 1$ for all $s$ in a fixed compact region. As a consequence, $\beta\Gamma(\beta) = \Gamma(1+\beta) \gg 1$ for all $|\beta| \le 1$, yielding the second assertion of the lemma. 
\section{Distribution of the sum-of-divisors to odd moduli: Proof of Theorem \ref{thm:sigmaodd}}\label{sec:ThmsigmaoddProof}
In the rest of this section, we abbreviate $\alpha(q)$ to $\alpha$. We shall make frequent use of the fact that $\alpha = \prod_{\ell \mid q} (1-1/(\ell-1)) \gg 1/\log_2 (3q)$ for all odd $q$. The following proposition allows us to give a rough estimate on the count of $n \le x$ for which $\sigma(n)$ is coprime to $q$, uniformly in odd moduli $q \le (\log x)^K$.  
\begin{prop}\label{prop:MathZProp2.1}
Fix $K>0$ and a multiplicative function $f$ for which there exists a non-constant polynomial $F \in \Z[T]$ satisfying $f(p)=F(p)$ for all primes $p$. If $x$ is sufficiently large and $q\le (\log{x})^{K}$ with $\alpha_F(q) \coloneqq \frac1{\phi(q)} \#\{u \bmod q: (uF(u), q)=1\}> 0$, then
\begin{equation}\label{eq:scourfieldestimate} \#\{n\le x: (f(n),q)=1\} = \frac{x}{(\log{x})^{1-\alpha_F(q)}} \exp(O((\log\log{(3q)})^{O(1)})). \end{equation}
\end{prop}
This is Proposition 2.1 in \cite{PSR23}; more precise results appear in work of Scourfield \cite{scourfield84, scourfield85}. By the above proposition, we obtain 
\begin{equation}\label{eq:sigmaqcoprim}
\#\{n\le x: (\sigma(n),q)=1\} = \frac{x}{(\log{x})^{1-\alpha}} \exp(O((\log\log{(3q)})^{O(1)})), 
\end{equation}
uniformly in odd $q \le (\log x)^K$. Since $\alpha \gg 1/\log_2(3q)$, this shows that the second of the two assertions of Theorem \ref{thm:sigmaodd} is an immediate consequence of the first, so we need to show the first assertion of the theorem. For this purpose, we shall also need the following estimate \cite[Lemma 2.4]{PSR23} on the sum of reciprocals of the primes at which a given polynomial is coprime to a modulus $q$ that varies in a wide range.
\begin{lem}\label{lem:primesum} Let $F(T)\in \Z[T]$ be a fixed nonconstant polynomial. For each positive integer $q$ and each real number $x\ge 3q$, \[ \sum_{p \le x} \frac{\1_{\gcd(F(p),q)=1}}{p} = \alpha _F(q)\log_2{x} + O((\log\log{(3q)})^{O(1)}),  \]
where $\alpha_F(q)$ is as defined in Proposition \ref{prop:MathZProp2.1}.
\end{lem} 
To establish the first assertion of Theorem \ref{thm:sigmaodd}, we set $y \coloneqq \exp((\log x)^{\epsilon/2})$ and $z \coloneqq x^{1/\log_2 x}$. From the count of $n \le x$ having $\sigma(n) \equiv a \pmod q$, we first remove those that are either $z$-smooth or have a repeated prime factor exceeding $y$. By well-known results on smooth numbers (for instance \cite[Theorem 5.13\text{ and }Corollary 5.19, Chapter III.5]{tenenbaum15}), the total contribution of such $n$ is $\ll \Psi(x, z) + x/y \ll x/(\log x)^{(1+o(1))\log_3 x}$, which is negligible in comparison to the error term in the statement of Theorem \ref{thm:sigmaodd}. 

Among the surviving $n$, we now remove those that have $P_2(n) \le y$. Any such $n$ can be written as $n = mP$, where $P = P(n)>z$, $m$ is $y$-smooth and $\sigma(n) = \sigma(m) \sigma(P) = \sigma(m) (P+1)$. The condition $\sigma(n) \equiv a \pmod q$ forces $\sigma(m)$ to be coprime to $q$ and, for each choice of $m$, restricts $P \in (z, x/m]$ to at most one coprime residue class modulo $q$. Hence for each choice of $m$, there are $\ll x/\phi(q)m\log(z/q) \ll x\log_2 x/\phi(q)m\log x$ many possible choices of $P$, by the Brun-Titchmarsh theorem. Consequently, the total contribution of the surviving $n \le x$ that have $P_2(n) \le y$ is
\begin{equation*}
\begin{split}
\ll \frac{x\log_2 x}{\phi(q)\log x}\sum_{\substack{m:\, P^+(m) \le y}} \frac{\bbm_{(\sigma(m), q)=1}}m &\ll \frac{x\log_2 x}{\phi(q)\log x} \exp\left(\sum_{p \le y} \frac{\bbm_{(p+1, \, q)=1}}p\right)\\
&\ll \frac{x\log_2 x}{\phi(q)(\log x)^{1-\alpha\epsilon/2}} \exp((\log_2 (3q))^{O(1)}) \ll \frac{x}{\phi(q)(\log x)^{1-2\alpha\epsilon/3}}.
\end{split}
\end{equation*}
Here we have estimated the sum $\sum_{p \le y} \bbm_{(p+1, \, q)=1}/p$ using Lemma \ref{lem:primesum} (on the polynomial $F(T) \coloneqq T+1$) and recalled that $\alpha \gg 1/\log_2(3q) \gg 1/\log_3 x$ for all odd $q \le (\log x)^K$. Collecting estimates, we have so far shown that
\begin{equation*}
\begin{split}
\sum_{\substack{n \le x\\\sigma(n) \equiv a \pmod q}} 1 &= \sum_{\substack{n \le x\\P(n)>z, \, P_2(n)>y\\p>y \implies p^2 \nmid n}} \bbm_{\sigma(n) \equiv a \pmod q} + O\left(\frac{x}{\phi(q)(\log x)^{1-2\alpha\epsilon/3}}\right)\\
&= \sum_{\substack{n \le x\\P(n)>z, \, P_2(n)>y\\p>y \implies p^2 \nmid n}} \phiqrec\sum_{\chi \bmod q} ~ \chibara \chisign + O\left(\frac{x}{\phi(q)(\log x)^{1-2\alpha\epsilon/3}}\right),
\end{split}
\end{equation*}
where in the second line above, we have used the orthogonality of the Dirichlet characters mod $q$ to detect the congruence $\sigma(n) \equiv a \pmod q$. With $\chi_{0, q}$ denoting the principal character modulo $q$, we may thus isolate the contribution of $\chi_{0, q}$ to obtain
\begin{equation}\label{eq:Postfirstremovals}
\begin{split}
\sum_{\substack{n \le x\\\sigma(n) \equiv a \pmod q}} 1 = 
\phiqrec&\sum_{\substack{n \le x\\(\sigma(n), q)=1}} 1\\ &+ \phiqrec\sum_{\chi \ne \chi_{0, q} \bmod q} \chibara\sum_{\substack{n \le x\\P(n)>z, \, P_2(n)>y\\p>y \implies p^2 \nmid n}} \chisign + O\left(\frac{x}{\phi(q)(\log x)^{1-2\alpha\epsilon/3}}\right).
\end{split} 
\end{equation}
The second outer sum on the right hand side above is over the non-trivial characters $\chi$ mod $q$, and we have adapted our previous arguments to observe that there are $O({x}/{(\log x)^{1-2\alpha\epsilon/3}})$ many $n \le x$ satisfying $(\sigma(n), q)=1$ but failing at least one of the three conditions below:
\begin{itemize}
    \item[(i)] $P(n)>z,$
    \item[(ii)] $p>y \implies p^2 \nmid n,$ 
    \item[(iii)] $P_2(n)>y$.  
\end{itemize}
Indeed, any $n$ satisfying conditions (i) and (ii) but failing condition (iii) is of the form $mP$, where $P^+(m) \le y$, $P = P^+(n) \in (z, x/m]$ and $\sigma(n) = \sigma(m)(P+1)$. As such, we must have $(\sigma(m), q)=1$, and the number of $P$ given $m$ is $\ll x/m\log z \ll x\log_2 x/m\log x$. Summing this expression over all possible $m$ yields the observed bound. 

In order to estimate the inner sums of $\chi(\sigma(n))$ occurring in \eqref{eq:Postfirstremovals}, we start by modifying some of our initial arguments in the proof of Theorem \ref{thm:mainbound}. Any $n$ with $P(n)>z$, $P_2(n)>y$ and without any repeated prime factor exceeding $y$ can be uniquely written in the form $mP_j \cdots P_1$ for some $j \ge 2$, where $P_1 = P(n) > z$ and $P(m) \le y < P_j < \dots < P_1$. As such, we have 
$$\sum_{\substack{n \le x\\P(n)>z, \, P_2(n)>y\\p>y \implies p^2 \nmid n}} \chisign = \sum_{j \ge 2} \sum_{\substack{m \le x\\P(m) \le y}} \chi(\sigma(m)) \sum_{\substack{P_1, \dots, P_j\\P_j \cdots P_1 \le x/m\\P_1>z, ~y<P_j< \dots < P_1}} \chi(P_1+1) \cdots \chi(P_j+1).$$ 
Define $\rho_\chi \coloneqq \phiqrec \sum_{v \bmod q} \chi_{0, q}(v) \chi(v+1)$. By the Siegel-Walfisz Theorem, we see that
\begin{equation}\label{eq:MainHypoCheck_SW}
\begin{split}
\sum_{y < p \le Y} \chi(p+1) &= \sum_{\substack{v \bmod q\\(v, q)=1}} \chi(v+1) \sum_{\substack{y < p \le Y\\ p \equiv v \pmod q}} 1\\ 
&= \sum_{\substack{v \bmod q\\(v, q)=1}} \chi(v+1) \left\{\phiqrec \sum_{\substack{y < p \le Y}} 1 + O\left(Y \exp(-C_0 \sqrt{\log y})\right) \right\}\\
&= \rho_\chi(\pi(Y) - \pi(y)) + O(\phi(q) Y \exp(-C_0\sqrt{\log y})),
\end{split}
\end{equation}
where $C_0$ is a constant depending at most on $K$. Proceeding as in the proof of Theorem \ref{thm:mainbound}, we successively remove $\chi(P_1+1), \dots, \chi(P_j+1)$, with the input from \eqref{eq:mainhypo} replaced by the above estimate. This leads us to
\begin{equation}\label{eq:Sigmaodd_ChisigFirstEstim}
\begin{split}
\sum_{\substack{n \le x\\P(n)>z, \, P_2(n)>y\\p>y \implies p^2 \nmid n}} \chisign = \sum_{j \ge 2} \frac{(\rho_\chi)^j}{(j-1)!} \sum_{\substack{m \le x\\P(m) \le y}} \chi(\sigma(m)) &\sum_{\substack{P_1, \dots, P_j\\P_1>z, ~ P_j \cdots P_1 \le x/m\\P_2, \dots , P_j \in (y, P_1) \text{ distinct }}} 1\\ &+ O(x \exp(-C_1 \sqrt{\log y})),
\end{split}
\end{equation}
for some constant $C_1>0$ depending at most on $K$. 

Now the main term in the display above is absolutely bounded by 
\begin{align}\allowdisplaybreaks\label{eq:SigmaOdd_MainFinal1}
|\rho_\chi|^2&\sumjT \frac{|\rho_\chi|^{j-2}}{(j-1)!} \sum_{\substack{m \le x\\P(m) \le y\\(\sigma(m), q)=1}}  \sum_{\substack{P_2, \dots, P_j \in (y, x)\\P_2 \cdots P_j \le x/mz}} ~ \sum_{z < P_1 \le x/mP_2 \cdots P_j} 1\\ \nonumber
&\ll \frac{|\rho_\chi|^2 x}{\log z}\sumjT \frac{|\rho_\chi|^{j-2}}{(j-1)!} \sum_{\substack{m \le x\\P(m) \le y\\(\sigma(m), q)=1}}  \frac1m \sum_{\substack{P_2, \dots, P_j \in (y, x)}} \frac1{P_2 \cdots P_j}\\ \nonumber
&\ll |\rho_\chi|^2\frac{x (\log_2 x)^2}{\log x} \sumjT \frac{(|\rho_\chi| \log_2 x)^{j-2}}{(j-2)!} \sum_{\substack{m \le x\\P(m) \le y\\(\sigma(m), q)=1}}  \frac1m \ll |\rho_\chi|^2\frac{x (\log_2 x)^2}{(\log x)^{1-|\rho_\chi|}} \exp\Bigg(\sum_{\substack{p \le y\\(p+1, \, q)=1}} \frac1p\Bigg). \label{eq:SigmaOdd_MainTerm1} 
\end{align}
Invoking Lemma \ref{lem:primesum} to estimate the last sum in the above display, we obtain
\begin{equation}\label{eq:chisign_Oddq1}
\sum_{\substack{n \le x\\P(n)>z, \, P_2(n)>y\\p>y \implies p^2 \nmid n}} \chisign \ll |\rho_\chi|^2 \frac x{(\log x)^{1-|\rho_\chi|-2\alpha\epsilon/3}} + x \exp(-C_1 \sqrt{\log y}). 
\end{equation}
In order to be able to make use of this bound, we need to better understand the $|\rho_\chi|$. To this end, given a nontrivial character $\chi \bmod q$ we let $\condofchi$ denote its conductor, so that $\condofchi \mid q$ and $\condofchi>1$. We can write $\chi$ uniquely in the form $\prod_{\ell^e \parallel q} \chi_\ell$, with $\chi_\ell$ denoting a character mod $\ell^e$ that is nontrivial precisely when $\ell \mid \condofchi$. Now $\phi(q)\rho_\chi = \prod_{\ell^e \parallel q} S_{\chi, \ell}$, where for each prime power $\ell^e \parallel q$,
\begin{equation}\label{eq:Schiell}
S_{\chi, \ell} \coloneqq \sum_{v \bmod \ell^e} \chi_{0, \ell}(v) \chi_\ell(v+1) = \sum_{\substack{v \bmod{\ell^e}\\(v, \ell)=1}} \chi_\ell(v+1) = \sum_{u \bmod{\ell^e}} \chi_\ell(u) - \sum_{\substack{u \bmod{\ell^e}\\u \equiv 1 \pmod \ell}} \chi_\ell(u).
\end{equation}
Here $\chi_{0, \ell}$ denotes the trivial character mod $\ell^e$ and we have noted that as $v$ runs over all the coprime residues mod $\ell^e$, the expression $v+1$ runs over all the residues mod $\ell^e$ except for those congruent to $1$ mod $\ell$. The first sum above is $\bbm_{\chi_\ell = \chi_{0, \ell}} \phi(\ell^e)$. To evaluate the second sum, we consider a primitive root $g$ mod $\ell^e$ (which exists as $\ell$ is odd), and observe that the residues $\{u \bmod{\ell^e}: u \equiv 1 \pmod\ell\}$ are a permutation of the residues $\{g^{(\ell-1)k} \bmod{\ell^e}: 0 \le k<\ell^{e-1}\}$. Hence 
\begin{equation}\label{eq:chiell_along_1modell}
\sum_{\substack{u \bmod{\ell^e}\\u \equiv 1 \pmod \ell}} \chi_\ell(u) = \sum_{0 \le k < \ell^{e-1}} \chi_\ell(g^{\ell-1})^k = \bbm_{(\chi_\ell)^{\ell-1} = \chi_{0, \ell}} ~ \ell^{e-1} = \bbm_{\cond(\chi_\ell) \mid \ell} ~ \ell^{e-1},
\end{equation}
with $\cond(\chi_\ell)$ denoting the conductor of $\chi_\ell$. Altogether, we find that 
$$S_{\chi, \ell} = \bbm_{\chi_\ell = \chi_{0, \ell}} \phi(\ell^e) - \bbm_{\cond(\chi_\ell) \mid \ell} \ell^{e-1} = \bbm_{\cond(\chi_\ell) \mid \ell} \ell^{e-1}\left(\bbm_{\ell \nmid \condofchi} (\ell-1) - 1\right),$$
leading to 
\begin{equation}\label{eq:rhochi}
\rho_\chi = \prod_{\ell^e \parallel q} \frac{S_{\chi, \ell}}{\phi(\ell^e)} = \bbm_{\condofchi \text{ squarefree }} \prod_{\ell^e \parallel q} \left(\bbm_{\ell \nmid \condofchi} - \frac1{\ell-1}\right) = \bbm_{\condofchi \text{ squarefree }} \frac{(-1)^{\omega(\condofchi)} \alpha}{\prod_{\ell \mid \condofchi} (\ell-2)}.
\end{equation}
If $3 \mid q$, let $\psi$ denote the unique character mod $q$ induced by the nontrivial character mod $3$. Then for any nonprincipal character $\chi \ne \psi$ mod $q$ for which $\rho_\chi \ne 0$, its conductor $\condofchi$ is divisible by a prime at least $5$, so that $|\rho_\chi| \le \alpha/3$ by \eqref{eq:rhochi}. As such, \eqref{eq:chisign_Oddq1} yields for all such $\chi$,
\begin{equation}\label{eq:chisign_Oddq2_NonExcp}
\sum_{\substack{n \le x\\P(n)>z, \, P_2(n)>y\\p>y \implies p^2 \nmid n}} \chisign \ll |\rho_\chi|^2 \frac x{(\log x)^{1-(1/3+2\epsilon/3)\alpha}} + x \exp(-C_1 \sqrt{\log y}). 
\end{equation}
Since there are exactly $\prod_{\ell \mid d} (\ell-2)$ primitive characters modulo any squarefree integer $d$, equation \eqref{eq:rhochi} yields
\begin{equation*}
\begin{split}
\sum_{\chi \bmod q} |\rho_\chi|^2 \le \alpha^2\sum_{\substack{d \mid q\\d\text{ squarefree}}} \frac1{\prod_{\ell \mid d} (\ell-2)^2} \sum_{\substack{\chi \bmod q\\\condofchi = d}} 1 \le \alpha^2\sum_{\substack{d \mid q\\d\text{ squarefree}}} \frac1{\prod_{\ell\mid d} (\ell-2)} \le \alpha^2 \prod_{\ell \mid q} \frac{\ell-1}{\ell-2} = \alpha.
\end{split}
\end{equation*}
(This may be compared with the bound placed on the averages $\frac1{\phi(q)} \sumvmodq \chi_{0, q}(v) \chi(v-1)$ in the proof of the analogous Theorem 1.3 in \cite{PSR23A}.)

Summing the bound \eqref{eq:chisign_Oddq2_NonExcp} over all nontrivial characters $\chi \ne \psi$ mod $q$, and plugging the resulting bound into \eqref{eq:Postfirstremovals}, we obtain 
\begin{equation}\label{eq:ReducnToExcpChar}
\begin{split}
\sum_{\substack{n \le x\\\sigma(n) \equiv a \pmod q}} 1 &= \phiqrec\sum_{\substack{n \le x\\(\sigma(n), q)=1}} 1 + \frac{\bbm_{3 \mid q} \overline{\psi}(a)}{\phi(q)}\sum_{\substack{n \le x\\P(n)>z, \, P_2(n)>y\\p>y \implies p^2 \nmid n}} \psi(\sigma(n)) + O\left(\frac{x}{\phi(q)(\log x)^{1-\alpha(1/3+\epsilon)}}\right)\\
&= \phiqrec\sum_{\substack{n \le x\\(\sigma(n), q)=1}} 1 + \frac{\bbm_{3 \mid q} \overline{\psi}(a)}{\phi(q)}\sum_{n \le x} \psi(\sigma(n)) + O\left(\frac{x}{\phi(q)(\log x)^{1-\alpha(1/3+\epsilon)}}\right).
\end{split}
\end{equation}
Here in passing to the second line above, we have recalled our previous bound on the count of $n \le x$ having $\sigma(n)$ coprime to $q$ but failing one of the conditions (i)--(iii) in the discussion following \eqref{eq:Postfirstremovals}. 

Note that the last equality in \eqref{eq:ReducnToExcpChar} already establishes the first assertion of the theorem for moduli $q$ coprime to $6$. To complete the proof of the theorem, it thus remains to only consider the odd moduli $q$ divisible by $3$ and deal with the sum of $\psi(\sigma(n))$ occurring in \eqref{eq:ReducnToExcpChar}. This is where we can directly apply Theorem \ref{thm:mainbound}. Indeed, since $3 \mid q$, we find that a prime $p$ satisfies $(p+1, \, q)=1$ only if $p=3$ or $p \equiv 1 \pmod 3$. Thus for all $p>3$, we have $\psi(\sigma(p)) = \psi(p+1) = -\bbm_{(p+1, \, q)=1}$. A straightforward calculation (analogous to \eqref{eq:MainHypoCheck_SW}) by means of the Siegel-Walfisz theorem shows that the multiplicative function $f(n) \coloneqq \psi(\sigma(n))$ satisfies the hypothesis \eqref{eq:mainhypo} with $y, z$ as chosen in the beginning of the section and with $\varrho \coloneqq -\frac1{\phi(q)}\sum_{\substack{v \bmod q\\(v(v+1), q)=1}} 1 = -\alpha$, $M \coloneqq \phi(q)$, $\err(y) \coloneqq \exp(-C_0 \sqrt{\log y})$ (for some absolute constant $C_0>0$). Since by lemma $\eqref{lem:primesum}$, $\sum_{p \le y} |f(p)|/p = \alpha \log_2 y + O((\log_2 (3q))^{O(1)})$, we deduce that 
$$\sum_{n \le x} \psi(\sigma(n)) \ll \frac{x(\log_2 x)^{1+\alpha}}{(\log x)^{1+\alpha(1-\epsilon)}} \exp((\log_2 (3q))^{O(1)}) \ll \frac{x}{\log x},$$ 
and substituting this into \eqref{eq:ReducnToExcpChar} completes the proof of the theorem.
\begin{rmk}
To substantiate our comment (made after the paragraph following the statement of Theorem \ref{thm:sigmaevensqfree}) on the suggested optimality of the exponent $1/3$ in the error term of Theorem \ref{thm:sigmaodd}, note that by \eqref{eq:rhochi}, we have $\rho_\chi = \alpha/3$ for $\condofchi = 15$ (in the case when $15 \mid q$).
\end{rmk}
\section{Technical Preparation for Theorems \ref{thm:sigmaeven} and \ref{thm:sigmaevensqfree}}
In order to establish Theorems \ref{thm:sigmaeven} and \ref{thm:sigmaevensqfree}, we will need to study the averages 
$$\eta_\chi \coloneqq \eta_\chi(q) \coloneqq \phiqrec \sum_{v \bmod q} \chi_{0, q}(v) \chi(v^2+v+1)$$
which will play the roles of the averages $\rho_\chi$ that came up in the previous section. The following proposition will provide the key information on these averages that will prove to be crucial in our arguments. 
\begin{prop}\label{prop:etachiBounds}
There exists a set $\mcS$ of eighteen (fixed) squarefree positive integers coprime to $6$ such that for all sufficiently large integers $q$ and all nonprincipal characters $\chi$ mod $q$, the following two properties hold true: 
\begin{enumerate}
\item[(i)] If $\condofchi \not\in \mcS$, then $|\eta_\chi| \le \alphatilq/4$.
\item[(ii)] If $\condofchi \in \mcS$, then $\Ree(\eta_\chi) \le \alphatilq/4$.
\end{enumerate}
In particular, we have $\Ree(\eta_\chi) \le \alphatilq/4$ for all nontrivial characters $\chi$ mod $q$, and $|\eta_\chi| \le \alphatilq/4$ for all but a bounded number of characters $\chi$ mod $q$.
\end{prop}
The following character sum bound, a special case of \cite[Theoem 1.1]{CLZ03}, will be useful to give a proof of Proposition \ref{prop:etachiBounds}.
\begin{lem}\label{lem:CLZ03}
Let $\ell$ be a prime at least $5$. Then for any integer $e \ge 2$ and any primitive character $\chi$ mod $\ell^e$, we have 
$$\left| \sum_{v \bmod \ell^e} \chi(v^2+v+1) \right| \le \ell^{e/2}.$$
\end{lem}
\begin{proof}[Proof of Proposition \ref{prop:etachiBounds}]
We start by factoring $\chi \eqqcolon \prod_{\ell^e \parallel q} \chi_\ell$, where each $\chi_\ell$ is as usual a character mod $\ell^e$. This allows us to factor $\etachi$ as $\prod_{\ell^e \parallel q} \etachiell$, where 
$$\etachiell \coloneqq \frac1{\phi(\ell^e)} \sum_{v \bmod \ell^e} \chi_{0, \ell}(v) \chi_\ell(v^2+v+1).$$
Since $\alphatil(\ell^e) = \frac1{\phi(\ell^e)} \#\{v \bmod \ell^e: (v(v^2+v+1), \ell)=1\}$, it is immediate that $|\etachiell| \le \alphatil(\ell^e) = \alphatil(\ell)$. Moreover, letting $e_\ell \coloneqq v_\ell(\condofchi)$ denote the exponent of the prime $\ell$ in the integer $\condofchi$, we see that $\condofchiell = \ell^{e_\ell}$, so if $e_\ell=0$ (i.e., $\ell \nmid \condofchi$), then $\chi_\ell = \chi_{0, \ell}$ and $\etachiell = \alphatil(\ell)$. 

Assume that $\chi_2$ is nontrivial, so that $e_2 \ge 2$. Letting $U_{2^{e_2}}$ denote the multiplicative group mod $2^{e_2}$ we observe that the map $U_{2^{e_2}} \rightarrow U_{2^{e_2}}: v \mapsto v^2+v+1$, being injective, is also bijective. As a consequence
$$\eta_{\chi, 2} = \frac1{\phi(2^{e_2})} \sum_{\substack{v \bmod 2^{e_2}\\(v, 2)=1}} \chi_2(v^2+v+1) = \frac1{\phi(2^{e_2})} \sum_{\substack{u \bmod 2^{e_2}\\(u, 2)=1}} \chi_2(u) = 0,$$
leading to $\eta_\chi = \eta_{\chi, 2} \prod_{\substack{\ell \mid q: ~ \ell>2}} \etachiell = 0$. Hence in the rest of the argument, it suffices to consider only those characters $\chi$ mod $q$ for which $\chi_2$ is trivial, that is, for which $\condofchi$ is not a multiple of $4$. Since $3 \nmid q$ and $\chi$ is nontrivial mod $q$, it must then be the case that $\chi_\ell \ne \chi_{0, \ell}$ for some prime $\ell \ge 5$ dividing $q$.

Consider a prime $\ell \ge 5$ dividing $q$ for which $\chi_\ell$ is nontrivial. Letting $\chi_\ell$ also denote the primitive chraracter mod $\ell^{e_\ell}$ that induces $\chi_\ell$ mod $\ell^e$, we find that 
\begin{align}\label{eq:etachi_FirstSimplif}\allowdisplaybreaks
\etachiell &= \phiellrec \cdot \ell^{e-e_\ell} \sumvmodelle \chiZellv \chiellvvO = \phielleellrec \sumvmodelle \chiZellv \chiellvvO\\ \nonumber
&=\phielleellrec \left\{\sumvmodelle \chiellvvO - \sum_{\substack{v \bmod \ell^{e_\ell}\\v \equiv 0 \pmod \ell}} \chiellvvO \right\}\\ \nonumber
&=\phielleellrec \left\{\sumvmodelle \chiellvvO - \bbm_{e_\ell=1} \right\}.    
\end{align}
In the last equality above, we have observed that the map $\{v \bmod \elleell: v \equiv 0 \pmod \ell\} \rightarrow \{u \bmod \elleell: u \equiv 1 \pmod \ell\} \colon v \mapsto v^2+v+1$ being an injection from one set to another of the same cardinality is also a bijection. By \eqref{eq:chiell_along_1modell}, this led to
$$\sum_{\substack{v \bmod \ell^{e_\ell}\\v \equiv 0 \pmod \ell}} \chiellvvO = \sum_{\substack{u \bmod \ell^{e_\ell}\\u \equiv 1 \pmod \ell}} \chi_\ell(u) = \bbm_{e_\ell=1} ~ \ell^{e_\ell-1} = \bbm_{e_\ell=1}.$$
If $\ell \ge 5$ and $e_\ell \ge 2$, then by Lemma \ref{lem:CLZ03}, we have 
\begin{equation}\label{eq:etachiellBound_eellge2}
|\etachiell| = \phielleellrec \left|\sumvmodelle \chiellvvO\right| \le \frac{\ell^{e_\ell/2}}{\phi(\elleell)} \le \frac{\ell^{1-e_\ell/2}}{\ell-1}.
\end{equation}
Hence if $v_\ell(\condofchi) = e_\ell \ge 2$ for some prime $\ell \ge 5$ dividing $q$, then 
$$\frac{|\etachi|}{\alphatilq} \le \frac{|\etachiell|}{\alphatil(\ell)} \le \frac{\ell^{1-e_\ell/2}}{\ell-1} \cdot \left(\frac{\ell-1}{\ell-3}\right)^{\bbm_{\ell \equiv 1 \pmod 3}} \le \frac14.$$
This shows that $|\etachi| \le \alphatilq/4$ for all characters $\chi$ mod $q$ whose conductor is not squarefree.

Consider now a nontrivial character $\chi$ mod $q$ whose conductor is squarefree. Then by \eqref{eq:etachi_FirstSimplif} and the Weil bounds (see for instance, \cite[Corollary 2.3]{wan97}), we find that 
\begin{equation}\label{eq:etachiellBound_eell=1}
|\etachiell| = \frac1{\phi(\ell)} \left|\sumvmodell \chiellvvO - 1\right| \le \frac{\ell^{1/2}+1}{\ell-1}.
\end{equation}
Applying this bound for each prime $\ell \ge 5$ dividing $q$, we find that 
$$\frac{|\etachi|}{\alphatilq} \le \prod_{\substack{\ell \mid q\\ \ell>2}} \mu_\ell, ~ ~ \text{ where } 
\mu_\ell \coloneqq \bbm_{\substack{\ell \equiv 1 \pmod 3}} \left(\frac{\ell^{1/2}+1}{\ell-3}\right) + \bbm_{\substack{\ell \ge 5, ~\ell \equiv 2 \pmod 3}} \left(\frac{\ell^{1/2}+1}{\ell-1}\right).$$
Note that $\mu_\ell \in (0, 1)$ for all primes $\ell \ge 5$. Since the functions $(\ell^{1/2}+1)/(\ell-3)$ and $(\ell^{1/2}+1)/(\ell-1)$ are both strictly decreasing, we observe the following cases in which $|\etachi|/\alphatilq \le 1/4$. For $i \in \{1, 2\}$, we let $\omega_i(r)$ denote the number of distinct primes dividing an integer $r$ that are congruent to $i$ mod $3$.
\begin{itemize}
\item If $P^+(\condofchi) \ge 29$, then $|\etachi|/\alphatilq \le \frac{29^{1/2}+1}{29-3} < 0.246$. 
\item If $\omega(\condofchi) \ge 4$, then one of the following three possibilities must hold:\\
(i) Either $\omegaOnefchi \ge 3$, in which case $\Modetachialphatil \le \mu_7 \mu_{13} \mu_{19} < 0.141$, \textbf{OR}\\
(ii) $\omegaTwofchi \ge 3$, in which case $\Modetachialphatil \le \mu_5 \mu_{11} \mu_{17} < 0.112$, \textbf{OR}\\
(iii) $\omegaOnefchi = \omegaTwofchi = 2$, in which case $\Modetachialphatil \le \mu_7 \mu_{13} \cdot \mu_5\mu_{11} < 0.147$. 
\item If $\omega(\condofchi)=3$ but $\condofchi$ does not lie in the set $A_0 \coloneqq \{5 \cdot 7 \cdot 11, 5 \cdot 7 \cdot 13\}$, then 
$$\Modetachialphatil \le \max\{\mu_5 \mu_{11} \mu_{13}, ~ \mu_{5} \mu_{7} \mu_{17}, ~ \mu_{7} \mu_{11} \mu_{13}, ~ \mu_{5} \mu_{7} \mu_{19}\} < 0.247.$$
\item If $\omega(\condofchi)=2$ but $\condofchi$ is not a member of the set $B_0 \coloneqq \{7 \cdot 13, 7 \cdot 19\} \cup \{5 \cdot 11, 5 \cdot 17\} \cup \{5 \cdot 7,  5 \cdot 13,  5 \cdot 19,  7 \cdot 11,  7 \cdot 17\}$, then 
$$\Modetachialphatil \le \max\{\mu_{13} \mu_{19}, ~ \mu_{11} \mu_{17}, ~ \mu_{5} \mu_{23}, ~  \mu_7 \mu_{23}, ~ \mu_{11} \mu_{13}\} < 0.241.$$
\end{itemize}
Hence, defining $\mcS$ to be the set $\{5, 7, 11, 13 ,17, 19, 23\} \cup A_0 \cup B_0$, we have shown that $|\etachi| \le \alphatilq/4$ for all characters $\chi$ mod $q$ whose conductor $\condofchi$ does not lie in the set $\mcS$. It thus only remains to show that for all $\chi$ mod $q$ with $\condofchi \in \mcS$, we have $\Ree(\etachi) \le \alphatilq/4$. For such characters, the identity \eqref{eq:etachi_FirstSimplif} shows that 
\begin{equation}
\begin{split}
\etachi &= \prod_{\substack{\ell \mid q\\\ell \nmid \condofchi}} \alphatil(\ell) \cdot \prod_{\ell \mid \condofchi} \left(\frac1{\phi(\ell)} \sum_{\substack{v \bmod \ell\\v \not\equiv 0 \bmod \ell}} \chi_\ell (v^2+v+1)\right)\\
&= \frac{\alphatil(q)}{\prod_{\substack{\ell \mid \condofchi\\\ell \equiv 1 \pmod 3}} (\ell-3) \cdot \prod_{\substack{\ell \mid \condofchi\\\ell \equiv 2 \pmod 3}} (\ell-1)} \sum_{\substack{v \bmod \condofchi\\\gcd(v, \condofchi)=1}} \psi (v^2+v+1),
\end{split}
\end{equation}
where $\psi \coloneqq \prod_{\ell \mid \condofchi} \chi_\ell$ denotes the primitive character mod $\condofchi$ inducing $\chi$. So we need only show that 
\begin{equation}\label{eq:Sagecode}
\max_{Q \in \mcS} ~ \frac1{\prod_{\substack{\ell \mid Q\\\ell \equiv 1 \pmod 3}} (\ell-3) \cdot  \prod_{\substack{\ell \mid Q\\\ell \equiv 2 \pmod 3}} (\ell-1)} ~ \max_{\substack{\psi \bmod Q\\\psi \text{ primitive}}} ~ \Ree \left(\sum_{\substack{v \bmod Q\\\gcd(v, Q)=1}} \psi (v^2+v+1)\right) \le \frac14.
\end{equation}
But $\mcS$ consists of only eighteen moduli, so this can be verified by a short Sage code. This completes the proof of Proposition \ref{prop:etachiBounds}. 
\end{proof} 
\begin{rmk}
The aforementioned Sage code actually shows that equality is attained in \eqref{eq:Sagecode} for $Q \in \{5, 7, 13, 35\} \subset \mcS$. In other words, for such $Q$, there exist primitive characters $\psi$ mod $Q$ for which 
$$\frac1{\prod_{\substack{\ell \mid Q\\\ell \equiv 1 \pmod 3}} (\ell-3) \cdot  \prod_{\substack{\ell \mid Q\\\ell \equiv 2 \pmod 3}} (\ell-1)} \cdot \Ree \left(\sum_{\substack{v \bmod Q\\\gcd(v, Q)=1}} \psi (v^2+v+1)\right) = \frac14.$$ 
As we shall see in the proof of Proposition \ref{prop:SigmaEvenConvenient} below, the averages $\etachi$ will play the roles of the parameter $\varrho$ from Theorem \ref{thm:mainbound}. This supports our previous comment on the expected optimality of the exponent $1/4$ in the error terms of Theorems \ref{thm:sigmaeven} and \ref{thm:sigmaevensqfree}. 
\end{rmk}
We shall also need the following analogue of Proposition \ref{prop:MathZProp2.1}, which gives a count for the main term in Theorems \ref{thm:sigmaeven} and \ref{thm:sigmaevensqfree}. In what follows, we abbreviate $\alphatilq$ to $\alphatil$.
\begin{lem}\label{lem:sigmaqcoprimEvenq}
Fix $K>0$. We have 
\begin{equation}\label{eq:sigman_evenq_estim} 
\#\{n\le x: (\sigma(n),q)=1\} = \frac{x^{1/2}}{(\log{x})^{1-\alphatil}} \exp(O((\log\log{(3q)})^{O(1)})), \end{equation}
uniformly in even $q \le (\log x)^K$ such that $3 \nmid q$. 
\end{lem}
\begin{proof}
Our key observation is that since $q$ is even, $\sigma(n)$ is coprime to $q$ if and only if $n$ is of the form $2^k m^2$ for some integer $k \ge 0$ and some odd integer $m$ satisfying $\gcd(\sigma(m^2), q)=1$; this follows from the fact that $\sigma(n) = \prod_{p^k \parallel n} \sigma(p^k) \equiv \prod_{p^k \parallel n: ~ p>2} ~ (k+1) \pmod 2$. In particular, if $n$ is of the form $r^2$ for some integer $r$, then $\gcd(\sigma(n), q)=1$. As such, the left hand side of \eqref{eq:sigman_evenq_estim} is no less than 
$$\sum_{\substack{r \le x^{1/2}\\(\sigma(r^2), q)=1}} 1 \ge \frac{x^{1/2}}{(\log{x})^{1-\alphatil}} \exp(O((\log\log{(3q)})^{O(1)})).$$
Here to write the last bound above, we have invoked Proposition \ref{prop:MathZProp2.1} with $f(n) \coloneqq \sigma(n^2)$ for which $F(T) \coloneqq T^2+T+1$ and $\alpha_F(q) = \alphatilq$. 

To obtain the upper bound, it suffices to show that the expression on the right hand side of \eqref{eq:sigman_evenq_estim} bounds (from above) the number of possible tuples $(k, m)$ of non-negative integers for which $m$ is odd, $2^k m^2 \le x$ and $\gcd(\sigma(m^2), q)=1$. The contribution of those tuples for which $k > 20 \log_2 x/\log 2$ is no more than
\begin{equation}\label{eq:SigmaevenCoprime_Largek}
\sum_{k>\frac{20 \log \log x}{\log 2}} ~ \sum_{m \le \sqrt{x/2^k}} 1 ~ \le ~ x^{1/2} ~ \sum_{k>\frac{20 \log_2 x}{\log 2}} ~ \frac1{2^{k/2}} ~ \ll ~ \frac{x^{1/2}}{(\log x)^{10}};
\end{equation}
this is negligible compared to the right hand side of \eqref{eq:sigman_evenq_estim}. On the other hand, if $k \le 20 \log_2 x/\log 2$, then $\sqrt{x/2^k} \ge x^{1/2}/(\log x)^{10} \ge x^{1/3}$. Consequently $q \le \big(\log \sqrt{x/2^k}\big)^{3K}$, and another application of Proposition \ref{prop:MathZProp2.1} shows that given $k$, the number of possible $m$ is at most
\begin{multline*}
\sum_{\substack{m \le \sqrt{x/2^k}\\(\sigma(m^2), q)=1}} 1 \ll \frac{\sqrt{x/2^k}}{(\log \sqrt{x/2^k})^{1-\alphatil}} \expOlogtwoqOOne\\ \ll \frac{x^{1/2}}{2^{k/2} (\log x)^{1-\alphatil}} \expOlogtwoqOOne.
\end{multline*}
Here we have noted that since $k \le 20 \log_2 x/\log x$, we have
$$\left(\log \sqrt{\frac x{2^k}}\right)^{1-\alphatil} = \left(\frac12\log x\right)^{1-\alphatil} \left(1+O\left(\frac k{\log x}\right)\right) = \left(\frac12\log x\right)^{1-\alphatil} \left(1+O\left(\frac{\log_2 x}{\log x}\right)\right).$$
Summing the bound in the above display over all $k \ge 0$ establishes the upper bound in \eqref{eq:sigman_evenq_estim}.
\end{proof}
To start proving Theorems \ref{thm:sigmaeven} and \ref{thm:sigmaevensqfree}, we set $y \coloneqq \exp((\log x)^{\epsilon/4})$ and $z \coloneqq x^{1/\log_2 x}$. 
We shall show that for even $q \le (\log x)^K$ and any coprime residue class $a$ mod $q$, the dominant contribution to the count of $n \le x$ satisfying $\sigma(n) \equiv a \pmod q$ comes from those $n$ which have sufficiently many large prime divisors. More precisely, these are the $n$ which have at least six prime factors exceeding $y$ counted with multiplicity.
\begin{prop}\label{prop:SigmaEvenConvenient}
Fix $K>0$ and $\epsilon \in (0, 1)$. We have 
\begin{equation}\label{eq:SigmaEvenConvenient}
\sum_{\substack{n \le x: ~ P_6(n)>y\\ \sigma(n) \equiv a \pmod q}} 1 = \phiqrec \sum_{\substack{n \le x\\ (\sigma(n), q)=1}} 1 +  O\left(\frac{x^{1/2}}{\phi(q)(\log x)^{1-\alphatil(q)(1/4+\epsilon)}}\right),
\end{equation}
uniformly in coprime residue classes $a$ mod $q$ to moduli $q \le (\log x)^K$ satisfying $\gcd(q, 6)=2$.
\end{prop}
\begin{proof}
We start by bounding the contribution of the $z$-smooth $n$ to the left hand side. By the observation made at the start of the proof of Lemma \ref{lem:sigmaqcoprimEvenq}, any such $n$ can be written in the form $2^k m^2$ for some $k \ge 0$ and some $z$-smooth odd $m$. The number of possibilities of $k$ given $m$ is $O(\log x)$. This in fact shows that 
\begin{equation}\label{eq:Sigmaeven_zsmooth}
\sum_{\substack{n \le x: ~ P(n) \le z\\\gcd(\sigma(n), q)=1}} 1 \le \sum_{\substack{m \le x^{1/2}\\ P(m) \le z}} \sum_{k \le \frac{\log x}{\log 2}} 1 \ll  \log x \sum_{\substack{m \le x^{1/2}\\ P(m) \le z}} 1 \ll \frac{x^{1/2}}{(\log x)^{(1/2+o(1)) \log_3 x}},
\end{equation}
where in the last step we have again invoked \cite[Theorem 5.13\text{ and }Corollary 5.19, Chapter III.5]{tenenbaum15}. The last expression above is negligible in comparison to the error term in \eqref{eq:SigmaEvenConvenient}.

Next, we bound the contribution of those $n$ which are divisible by the fourth power of a prime exceeding $y$. As before, any such $n$ can be written in the form $2^k m^2$ for some $k \ge 0$ and some odd $m \le \sqrt{x/2^k}$ where (this time) $m$ is divisible by the square of a prime exceeding $y$. Given $k$, the number of possibilities of $m$ is no more than
$$\sum_{p>y} \sum_{\substack{m \le \sqrt{x/2^k}\\p^2 \mid m}} 1 \le \sqrt{\frac{x}{2^k}} \sum_{p>y} \frac1{p^2} \ll \frac{\sqrt{x/2^k}}y.$$
Summing this bound over all $k \ge 0$, we find that 
\begin{equation}\label{eq:Sigmaeven_4thpower}
\sum_{\substack{n \le x: ~ (\sigma(n), q)=1\\\exists p>y \text{ s.t.} ~ p^4 \mid n}} 1 \ll \frac{\sqrt{x}}y,
\end{equation}
which is also negligible compared to the error term in \eqref{eq:SigmaEvenConvenient}.

By \eqref{eq:Sigmaeven_zsmooth} and \eqref{eq:Sigmaeven_4thpower}, we may thus ignore the contribution of those $n$ to the left hand side of \eqref{eq:SigmaEvenConvenient}, which are either $z$-smooth or are divisible by the fourth power of a prime exceeding $y$. In order to complete the proof of the proposition, it thus remains to show that
\begin{equation}\label{eq:Sigmaeeven_ConvFiltered}
\sum_{\substack{n \le x: ~ P_6(n)>y\\ P(n)>z; ~ p>y \implies p^4 \nmid n\\\sigma(n) \equiv a \pmod q}} 1 = \phiqrec \sum_{\substack{n \le x\\(\sigma(n), q)=1}} 1 +  O\left(\frac{x^{1/2}}{\phi(q)(\log x)^{1-\alphatil(q)(1/4+\epsilon)}}\right).
\end{equation}
To prove this estimate, we start by invoking the orthogonality of the Dirichlet characters mod $q$ to write 
\begin{equation}\label{eq:SigmaEven_Orthog1}
\sum_{\substack{n \le x: ~ P_6(n)>y\\ P(n)>z; ~ p>y \implies p^4 \nmid n\\\sigma(n) \equiv a \pmod q}} 1 = \phiqrec \sum_{\substack{n \le x: ~ P_6(n)>y\\ P(n)>z; ~ p>y \implies p^4 \nmid n\\\gcd(\sigma(n), q)=1}} 1 + \phiqrec \sum_{\chi \ne \chi_{0, q} \bmod q} \chibara \sum_{\substack{n \le x: ~ P_6(n)>y\\ P(n)>z; ~ p>y \implies p^4 \nmid n}} \chi(\sigma(n)).\\
\end{equation}
We remove the additional conditions in the first sum on the right hand side above. To begin with, we observe that up to a negligible error, we may ignore the condition $P_6(n)>y$: indeed, any $n \le x$ which violates this condition but satisfies all the other conditions in the first sum can be written in the form $2^k m^2$ where $m$ is odd but not $z$-smooth, has no repeated prime factor exceeding $y$, and satisfies $P_6(m^2) \le y$ and $\gcd(\sigma(m^2), q)=1$. As such, $m$ can be written in the form $rP$ for some prime $P>z$ and some odd $r$ coprime to $P$ which satisfies $P_2(r) \le y$ and $\gcd(\sigma(r^2), q)=1$. Altogether, we find that 
\begin{align}\allowdisplaybreaks \label{eq:P6nley_BoundPrep}
\sum_{\substack{n \le x: ~ P_6(n) \le y\\ P(n)>z; ~ p>y \implies p^4 \nmid n\\\gcd(\sigma(n), q)=1}} 1 &\le \sum_{k \ge 0} ~ \sum_{\substack{r \le x: ~ P_2(r) \le y\\(\sigma(r^2), q)=1}} ~\sum_{z<P \le \sqrt{x/2^k r^2}} 1 \nonumber \\ &\ll \frac{x^{1/2}}{\log z} \sum_{k \ge 0} \frac1{2^{k/2}} ~ \sum_{\substack{r \le x: ~ P_2(r) \le y\\(\sigma(r^2), q)=1}} \frac1r
\ll \frac{x^{1/2} \log_2 x}{\log x} \sum_{\substack{r \le x: ~ P_2(r) \le y\\(\sigma(r^2), q)=1}} \frac1r. 
\end{align}
Any $r$ counted in the above sum can be written in the form $AB$, where $P(B) \le y < P(A)$, so that $A$ is either $1$ or a prime and $\sigma(r^2) = \sigma(A^2) \sigma(B^2)$. Given $B$, the sum of $1/A$ over all possible $A$ is at most $1+\sum_{p \le x} 1/p \ll \log_2 x$. We infer that 
\begin{equation}\label{eq:Sum1r_(Sigmar2,q)=1}
\begin{split}
\sum_{\substack{r \le x: ~ P_2(r) \le y\\(\sigma(r^2), q)=1}} \frac1r &\ll (\log_2 x) \sum_{\substack{B \le x: ~ P(B) \le y\\(\sigma(B^2), q)=1}} \frac1B \ll (\log_2 x)  \exp\left(\sum_{p \le y} \frac{\bbm_{(p^2+p+1, \, q)=1}}p\right)\\ &\ll (\log x)^{\alphatil \epsilon/4} (\log_2 x) \expOlogtwoqOOne,
\end{split}     
\end{equation}
where in the last line above, we have invoked Lemma \ref{lem:primesum} on the polynomial $F(T) \coloneqq T^2+T+1$. Inserting the above bound into \eqref{eq:P6nley_BoundPrep}, we obtain
\begin{equation}\label{eq:P6nley_Bound}
\begin{split}
\sum_{\substack{n \le x: ~ P_6(n) \le y\\ P(n)>z; ~ p>y \implies p^4 \nmid n\\\gcd(\sigma(n), q)=1}} 1 \ll \frac{x^{1/2} (\log_2 x)^2}{(\log x)^{1-\alphatil \epsilon/4}} \expOlogtwoqOOne \ll \frac{x^{1/2}}{(\log x)^{1-\alphatil \epsilon/2}},
\end{split}
\end{equation}
whereupon from \eqref{eq:SigmaEven_Orthog1}, it follows that
\begin{multline}\label{eq:Sigmaeven_Orthog2}
\sum_{\substack{n \le x: ~ P_6(n)>y\\ P(n)>z; ~ p>y \implies p^4 \nmid n\\\sigma(n) \equiv a \pmod q}} 1 = \phiqrec \sum_{\substack{n \le x\\(\sigma(n), q)=1}} 1\\ + \phiqrec \sum_{\chi \ne \chi_{0, q} \bmod q} \chibara \sum_{\substack{n \le x\\P(n)>z, ~ P_6(n)>y\\ p>y \implies p^4 \nmid n}} \chi(\sigma(n)) + O\left(\frac{x^{1/2}}{\phi(q) (\log x)^{1-\alphatil \epsilon/2}}\right).
\end{multline}
Here we have also used \eqref{eq:Sigmaeven_zsmooth} and \eqref{eq:Sigmaeven_4thpower} respectively to remove the conditions ``$P(n)>z$" and ``$p>y \implies p^4 \nmid y$" occurring in the first sum on the right hand side of \eqref{eq:SigmaEven_Orthog1}.

In order to estimate the inner sums on $\chi(\sigma(n))$ in \eqref{eq:Sigmaeven_Orthog2}, we proceed analogously to our proof of Theorem \ref{thm:sigmaodd}. For a given nontrivial character $\chi$ mod $q$, any $n$ counted in the aforementioned sum can be uniquely written in the form $M P_1^2 \cdots P_j^2$ for some $j \ge 3$, some $y$-smooth $M$ and some primes $P_1, \dots, P_j$, which satisfy $P_1>z$ and $y<P_j < \dots <P_2 < P_1$. (Here the condition $j \ge 3$ is a consequence of $P_6(n)>y$.) Proceeding as in the proof of Theorem \ref{thm:mainbound}, and using the estimate 
$$\sum_{y<p \le Y} \chi(p^2+p+1) = \eta_\chi (\pi(Y) - \pi(y)) + O(\phi(q) Y \exp(-C_0 \sqrt{\log y}))$$
in place of \eqref{eq:MainHypoCheck_SW}, we obtain the following analogue of \eqref{eq:Sigmaodd_ChisigFirstEstim}
\begin{equation}\label{eq:SigmaEven_ChisigFirstEstim}
\begin{split}
\sum_{\substack{n \le x\\P(n)>z, ~ P_6(n)>y\\p>y \implies p^4 \nmid n}} \chisign = \sum_{j \ge 3} \frac{(\eta_\chi)^j}{(j-1)!} \sum_{\substack{M \le x\\P(M) \le y}} \chi(\sigma(M)) &\sum_{\substack{P_1, \dots, P_j\\P_1>z, ~ P_j \cdots P_1 \le \sqrt{x/M}\\P_2, \dots , P_j \in (y, P_1) \text{ distinct }}} 1\\ &+ O(x^{1/2} \exp(-C_1 \sqrt{\log y})),
\end{split}
\end{equation}
where $C_1 = C_1(K)$ is a constant. Bounding the main term above as in \eqref{eq:SigmaOdd_MainFinal1}, we deduce that
\begin{equation}\label{eq:ChiSign_evenq_Naive}
\sum_{\substack{n \le x\\P(n)>z, ~ P_6(n)>y\\p>y \implies p^4 \nmid n}} \chisign \ll |\eta_\chi|^3 ~ \frac{x^{1/2} (\log_2 x)^3}{(\log x)^{1-|\eta_\chi|}} \sum_{\substack{M \le x\\P(M) \le y\\(\sigma(M), q)=1}}  \frac1{M^{1/2}} + x^{1/2} \exp(-C_1 \sqrt{\log y}).
\end{equation}
To estimate the sum on $M$ above, we recall that, by the observation made at the start of the proof of Lemma \ref{lem:sigmaqcoprimEvenq}, any $M$ counted in the sum can be uniquely written in the form $2^k m^2$ for some $k \ge 0$ and some odd $y$-smooth $m$ satisfying $\gcd(\sigma(m^2), q)=1$. We thus obtain 
\begin{equation}\label{eq:Sigmaeven_MRecipSum}
\begin{split}
\sum_{\substack{M \le x\\P(M) \le y\\(\sigma(M), q)=1}}  \frac1{M^{1/2}} &\le \sum_{k \ge 0}  \frac1{2^{k/2}} \sum_{\substack{m \le \sqrt{x/2^k}\\P(m) \le y\\(\sigma(m^2), q)=1}}  \frac1m\\ &\ll \exp\left(\sum_{p \le y} \frac{\bbm_{(p^2+p+1, \, q)=1}}p\right) \ll (\log x)^{\alphatil \epsilon/4} \expOlogtwoqOOne.  
\end{split}
\end{equation}
Inserting this bound into \eqref{eq:ChiSign_evenq_Naive} yields
\begin{equation*}
\sum_{\substack{n \le x\\P(n)>z, ~ P_6(n)>y\\p>y \implies p^4 \nmid n}} \chisign \ll |\eta_\chi|^3 ~ \frac{x^{1/2} (\log_2 x)^3}{(\log x)^{1-|\eta_\chi| - \alphatil\epsilon/4}} \expOlogtwoqOOne + x^{1/2} \exp(-C_1 \sqrt{\log y}).
\end{equation*}
Assume that $\condofchi \not\in \mcS$, where $\mcS$ is the set of eighteen positive integers considered in Proposition \ref{prop:etachiBounds}. Then $|\etachi| \le \alphatil/4$, and we obtain, for all such characters $\chi$ mod $q$,
\begin{equation}\label{eq:ChiSign_evenq_Naive3}
\sum_{\substack{n \le x\\P(n)>z, ~ P_6(n)>y\\p>y \implies p^4 \nmid n}} \chisign \ll |\eta_\chi|^3 ~ \frac{x^{1/2}}{(\log x)^{1-\alphatil(1/4+\epsilon/2)}} + x^{1/2} \exp(-C_1 \sqrt{\log y}).
\end{equation}
Now from the computations in \eqref{eq:etachiellBound_eellge2} and \eqref{eq:etachiellBound_eell=1}, we see that for each nontrivial character $\chi$ mod $q$, we have (with $e_\ell \coloneqq v_\ell(\condofchi)$ as before),
$$|\etachi| = \prod_{\ell \mid q} |\etachiell| \le \prod_{\ell \mid \condofchi} \frac{\ell^{e_\ell/2}+1}{\phi(\elleell)} \le  \prod_{\ell \mid \condofchi} \ell^{-e_\ell/2} 
 \left(1+O\left(\frac1{\ell^{1/2}}\right)\right) \le \condofchi^{-1/2} \exp(O(\sqrt{\log q})).$$
Since there are no more than $d$ characters mod $q$ having conductor $d$, we obtain
\begin{equation*}
\begin{split}
\sum_{\chi \ne \chi_{0, q} \bmod q} |\etachi|^3 \le \sum_{d \mid q} d \cdot \frac1{d^{3/2}} \exp(O(\sqrt{\log q})) \le \exp(O(\sqrt{\log q})),
\end{split}
\end{equation*}
where we have noted that $\sum_{d \mid q} {d^{-1/2}} \le \prod_{\ell \mid q} (1+O(\ell^{-1/2})) \le \exp(O(\sqrt{\log q}))$. Summing the bound \eqref{eq:ChiSign_evenq_Naive3} over all nonprincipal characters $\chi$ mod $q$ having $\condofchi \not\in \mcS$, and invoking the bound on $\sum_{\chi \ne \chi_{0, q} \bmod q} |\etachi|^3$ obtained above, we find that
\begin{equation}\label{eq:SigmaEven_mcStotal}
\sum_{\substack{\chi \ne \chi_{0, q} \bmod q\\ \condofchi \not\in \mcS}} \left|\sum_{\substack{n \le x: ~ P(n)>z\\P_6(n)>y; ~ p>y \implies p^4 \nmid n}} \chi(\sigma(n))\right| \ll \frac{x^{1/2}}{(\log x)^{1-\alphatil(1/4+\epsilon)}}.
\end{equation}
It remains to consider the characters $\chi$ mod $q$ whose conductors lie in the set $\mcS$. For each such character, we may invoke \eqref{eq:P6nley_Bound}, \eqref{eq:Sigmaeven_zsmooth} and \eqref{eq:Sigmaeven_4thpower} to obtain 
$$\sum_{\substack{n \le x\\P_6(n)>y, ~ P(n)>z\\ p>y \implies p^4 \nmid n}} \chi(\sigma(n)) = \sum_{\substack{n \le x}} \chi(\sigma(n)) + O\left(\frac{x^{1/2}}{(\log x)^{1-\alphatil \epsilon/2}}\right).$$
Recalling the observation made at the start of the proof of Lemma \ref{lem:sigmaqcoprimEvenq} along with the bound \eqref{eq:SigmaevenCoprime_Largek}, we obtain
\begin{equation}\label{eq:Sigmaeven_MainboundApp_Prep}
\sum_{\substack{n \le x\\P_6(n)>y, ~ P(n)>z\\ p>y \implies p^4 \nmid n}} \chi(\sigma(n)) = \sum_{\substack{k \le \frac{20 \log \log x}{\log 2}}} \chi(\sigma(2^k)) ~ \sum_{m \le \sqrt{x/2^k}} \bbm_{2 \nmid m} ~\chi(\sigma(m^2)) ~ + ~ O\left(\frac{x^{1/2}}{(\log x)^{1-\alphatil \epsilon/2}}\right).
\end{equation}
Finally, we invoke Theorem \ref{thm:mainbound} on the multiplicative function $m \mapsto \bbm_{2 \nmid m} ~ \chi(\sigma(m^2))$ to bound each of the inner sums in the above display. Noting that $\sqrt{x/2^k} \ge x^{1/2}/(\log x)^{10} > z$ and that $\etachi$ plays the role of $\varrho$, we deduce that the sums on $m$ in the above display are all
\begin{multline*}
\ll \frac{\sqrt{x/2^k}}{(\log z)^{1-\Ree(\etachi)}} \left( (\log y)^{\alphatil} + \frac{(\log y)^{1+\alphatil}}{\log z}\right) \exp\left(\sum_{p \le y} \frac{\bbm_{(p^2+p+1, \, q)=1}}p\right)\\ \ll \frac{x^{1/2} (\log_2 x)^2}{2^{k/2}(\log x)^{1-\alphatil/4 - \alphatil\epsilon/2}} \expOlogtwoqOOne \ll \frac{x^{1/2}}{2^{k/2}(\log x)^{1-\alphatil(1/4 + \epsilon)}}. 
\end{multline*}
In the last line above, we have utilized the second assertion of Proposition \ref{prop:etachiBounds} (namely that $\Ree(\etachi) \le \alphatil/4$) in conjunction with Lemma \ref{lem:primesum}. Summing the above bound over all $k \ge 0$ and inserting into \eqref{eq:Sigmaeven_MainboundApp_Prep}, we obtain
$$\sum_{\substack{n \le x\\P_6(n)>y, ~ P(n)>z\\ p>y \implies p^4 \nmid n}} \chi(\sigma(n)) \ll \frac{x^{1/2}}{(\log x)^{1-\alphatil(1/4 + \epsilon)}}$$
for all characters $\chi$ mod $q$ having $\condofchi \in \mcS$. We use this bound for each of the $O(1)$ nontrivial characters $\chi$ mod $q$ having $\condofchi \in \mcS$ and use \eqref{eq:SigmaEven_mcStotal} to deal with the rest of the characters mod $q$. Inserting these bounds into \eqref{eq:Sigmaeven_Orthog2}, we obtain the desired estimate \eqref{eq:Sigmaeeven_ConvFiltered}, which completes the proof of the proposition. 
\end{proof}
\begin{rmk}
The proofs of Theorem \ref{thm:sigmaodd} and Proposition \ref{prop:SigmaEvenConvenient} substantiate the comments following Theorem \ref{thm:mainbound}. Note that $\varrho = -\alpha$ for the sum $\sum_{n \le x} \psi(\sigma(n))$ at the end of section \ref{sec:ThmsigmaoddProof}, so a direct application of \cite[Theorem 1.1]{PSR23A} (or the methods used to prove it) would be unable to detect the negative sign of $\varrho$ and would yield a bound on this sum which would have the same order of magnitude as the main term (by Lemma \ref{lem:sigmaqcoprimEvenq}). A similar phenomenon takes place in the proof of Proposition \ref{prop:SigmaEvenConvenient} for the sums of $\chi(\sigma(n))$ for the characters $\chi$ having conductors in the set $\mcS$. On the other hand, we cannot apply this paper's Theorem \ref{thm:mainbound} for all the nontrivial characters $\chi$ mod $q$, for if we did, then the terms with $(\log y)^{1+|\varrho|}$ in Theorem \ref{thm:mainbound} would culminate into a large error term that would stand in the way of achieving uniformity in $q$ up to (fixed) large powers of $\log x$. 
\end{rmk}
\section{Distribution of the sum-of-divisors function to general even moduli: Proof of Theorem \ref{thm:sigmaeven}}
We continue with $y$ and $z$ as defined in the previous section. Note that the right hand sides of \eqref{eq:sigmaeven} and \eqref{eq:SigmaEvenConvenient} are equal up to a negligible error: indeed any $n$ having $P_6(n) \le q$ also has $P_6(n) \le y$, so that by \eqref{eq:Sigmaeven_zsmooth}, \eqref{eq:Sigmaeven_4thpower} and \eqref{eq:P6nley_Bound}, the contribution of all such $n$ to the right hand side of \eqref{eq:sigmaeven} is absorbed in the error term. By Proposition \ref{prop:SigmaEvenConvenient}, it thus suffices to show that 
\begin{equation}\label{eq:SigmaEven_RemInconv}
\sum_{\substack{n \le x: ~ q<P_6(n) \le y\\ P(n)>z; ~ p>y \implies p^4 \nmid n\\\sigma(n)\equiv a \pmod q}} 1 \ll \frac{x^{1/2}}{\phi(q) (\log x)^{1-\alphatil \epsilon/2}}
\end{equation}
in order to complete the proof of Theorem \ref{thm:sigmaeven}. 

We write the left hand side of \eqref{eq:SigmaEven_RemInconv} as $\Sigma_1 + \Sigma_2$, where $\Sigma_1$ denotes the count of the $n$ contributing to the sum in \eqref{eq:SigmaEven_RemInconv} which are divisible by the fourth power of a prime exceeding $q$. 
First consider the contribution of the $n$ counted in $\Sigma_1$. As before, the coprimality of $\sigma(n)$ with $q$ guarantees that we can write $n = 2^k m^2$ for some $k \ge 0$ and some odd $m$ satisfying $P_6(m^2) \le y$. Since $n$ is divisible by the fourth power of a prime exceeding $q$, it follows that the squarefull part of $m$ (i.e., the largest squarefull divisor of $m$) is divisible by a prime exceeding $q$. Hence $m$ can be written in the form $rSP$, with $r, S, P$ being pairwise coprime and satisfying $P_2(r) \le y$, $P = P(n)>z$ and with $S>q^2$ being squarefull. Altogether, we find that 
\begin{equation}\label{eq:SigmaOneBound}
\begin{split}
\Sigma_1 &\le \sum_{\substack{n \le x: ~ P_6(n) \le y\\ P(n)>z; ~ p>y \implies p^4 \nmid n\\\exists p>q: ~ p^4|n\\\sigma(n)\equiv a \pmod q}} 1 \le \sum_{k \ge 0} ~ \sum_{\substack{r \le x^{1/2}: ~ P_2(r) \le y\\p>y \implies p^2 \nmid r\\(\sigma(r^2), q)=1}} ~ \sum_{\substack{S>q^2\\S\text{ squarefull}}} ~ \sum_{z < P \le x^{1/2}\big/2^{k/2} rS} 1\\
&\ll \frac{x^{1/2}}{\log z} \sum_{k \ge 0} \frac1{2^{k/2}} ~ \sum_{\substack{r \le x^{1/2}: ~ P_2(r) \le y\\(\sigma(r^2), q)=1}} \frac1r ~ \sum_{\substack{S>q^2\\S\text{ squarefull}}} \frac1S \ll \frac{x^{1/2} \log_2 x}{q \log x} \sum_{\substack{r \le x^{1/2}: ~ P_2(r) \le y\\(\sigma(r^2), q)=1}} \frac1r \ll \frac{x^{1/2}}{q (\log x)^{1-\alphatil\epsilon/2}},
\end{split}
\end{equation}
where have used \eqref{eq:Sum1r_(Sigmar2,q)=1} and the standard bound $\sum_{S>q^2 \text{  squarefull }} 1/S \ll 1/q$. Since the last expression above is absorbed in the right hand side of \eqref{eq:SigmaEven_RemInconv}, it remains to show that the same is true for the sum $\Sigma_2$. 

Now any $n$ counted in $\Sigma_2$ has $P_6(n)>q$ but is not divisible by the fourth power of a prime exceeding $q$. Invoking the observation at the start of the proof of Lemma \ref{lem:sigmaqcoprimEvenq}, we find that any such $n$ can be written in the form $2^k r^2 (P_1 P_2 P_3)^2$, where $k \ge 0$, $P_2(r) \le y$, and $P_1$, $P_2$, $P_3$ are primes satisfying $P_1 = P(n)>z$, $q<P_3<P_2<P_1$, and $\sigma(n) = \sigma(2^k) \sigma(r^2) \prod_{j=1}^3 (P_j^2 + P_j + 1)$. Given $k$ and $r$, the congruence $\sigma(n) \equiv a \pmod q$ forces $(P_1, P_2, P_3)$ mod $q$ $\in \mathcal V_q(a\sigma(2^k r^2)^{-1})$, where $u^{-1}$ denotes the multiplicative inverse of a coprime residue $u$ mod $q$, and for any coprime residue $w$ mod $q$, we have defined
$$\mathcal V_q(w) \coloneqq \left\{(v_1, v_2, v_3) \in U_q^3: ~ \prod_{j=1}^3 (v_j^2 + v_j + 1) \equiv w \pmod q \right\}.$$
(Recall that $U_q$ denotes the group of units modulo $q$.) Given $k, r$ and $(v_1, v_2, v_3) \in \mathcal V_q(a\sigma(2^k r^2)^{-1})$, we bound the number of possible choices of $P_1, P_2, P_3$ which satisfy $(P_1, P_2, P_3) \equiv (v_1, v_2, v_3) \bmod q$. Given $P_2$ and $P_3$, the number of possible $P_1 \in (z, x^{1/2}/2^{k/2} r P_2 P_3]$ satisfying $P_1 \equiv v_1 \pmod q$ is, by the Brun-Titichmarsh inequality, 
$$\ll \frac{x^{1/2}/2^{k/2} r P_2 P_3}{\phi(q) \log(z/q)} \ll \frac{x^{1/2} \log_2 x}{\phi(q) \log x} \cdot \frac1{2^{k/2} r P_2 P_3}.$$
We now sum this bound over all $P_2, P_3 \in (q, x]$ satisfying $P_2 \equiv v_2 \pmod q$ and $P_3 \equiv v_3 \pmod q$. By Brun-Titchmarsh and partial summation, we have
$$\sum_{\substack{q<p \le x\\p \equiv v \pmod q}} \frac1p \ll \frac{\log_2 x}{\phi(q)}.$$
Applying this to the sums on $P_2$ and $P_3$, we find that given $k, r$ and $(v_1, v_2, v_3) \in \mathcal V_q(a\sigma(2^k r^2)^{-1})$, the number of possible $P_1, P_2, P_3$ satisfying $(P_1, P_2, P_3) \equiv (v_1, v_2, v_3) \bmod q$ is 
$$\ll \frac1{\phi(q)^3} \cdot \frac{x^{1/2} (\log_2 x)^3}{2^{k/2} r \log x}.$$
Now let $V_q \coloneqq \max\{\#\mathcal V_q(w): w \in U_q\}$. Summing the above bound over all $(v_1, v_2, v_3) \in \mathcal V_q(a\sigma(2^k r^2)^{-1})$, and subsequently over all $k$ and $r$, we obtain
\begin{equation}\label{eq:Sigma2_ReducntoOrthog}
\Sigma_2 \ll \frac{V_q}{\phi(q)^3} \cdot \frac{x^{1/2} (\log_2 x)^4}{(\log x)^{1-\alphatil\epsilon/4}} \expOlogtwoqOOne.   
\end{equation}
To bound $V_q$, we consider an arbitrary coprime residue $w$ mod $q$, and note that $\#\mathcal V_q(w) = \prod_{\ell^e \parallel q} \#\Vellew$ by the Chinese Remainder Theorem. Moreover, by orthogonality, 
\begin{equation*}
\begin{split}
\#\Vellew &= \sum_{v_1, v_2, v_3 \bmod \ell^e} \chi_{0, \ell} (v_1 v_2 v_3) \cdot \phiellrec \sum_{\chi \bmod \ell^e} \overline\chi(w) \chi\left(\prod_{j=1}^3 (v_j^2+v_j+1) \right)\\
&= \frac{(\alphatil(\ell) \phi(\ell^e))^3}{\phi(\ell^e)} \left\{1 + \frac1{(\alphatil(\ell) \phi(\ell^e))^3} \sum_{\chi \ne \chiZell \bmod \ell^e} \overline\chi(w) \left(\sum_{v \bmod \ell^e} \chiZell(v) \chi(v^2+v+1)\right)^3 \right\}.
\end{split}
\end{equation*}
Given $\chi \ne \chiZell$ mod $\ell^e$, let $\ell^{e_0}$ denote the conductor of $\chi$, so that $e_0 \in \{1, \dots, \ell^e\}$. Then with $\chi$ also denoting the primitive character mod $\ell^{e_0}$ inducing $\chi$, the computations and arguments in \eqref{eq:etachi_FirstSimplif}, \eqref{eq:etachiellBound_eellge2} and \eqref{eq:etachiellBound_eell=1} reveal that if $\ell \ge 5$, then
\begin{equation*}
\begin{split}
\left|\sum_{v \bmod \ell^e} \chiZell(v) \chi(v^2+v+1)\right| = \ell^{e-e_0} \left|\sum_{v \bmod \ell^{e_0}} \chi(v^2+v+1) - \bbm_{e_0=1}\right| \ll \ell^{e-e_0} \cdot \ell^{e_0/2} \ll \ell^{e-e_0/2}. 
\end{split}
\end{equation*}
Since there are at most $\phi(\ell^{e_0})$ characters mod $\ell^e$ with conductor $\ell^{e_0}$, we obtain
$$\#\Vellew = \frac{(\alphatil(\ell) \phi(\ell^e))^3}{\phi(\ell^e)} \left\{1 + O \left(\frac1{\phi(\ell^e)^3} \sum_{1 \le e_0 \le e} \phi(\ell^{e_0}) ~ (\ell^{e-e_0/2})^3\right) \right\} \le \frac{\phi(\ell^e)^3}{\phi(\ell^e)} \left\{1 + O \left(\frac1{\ell^{1/2}}\right) \right\},$$
where we have recalled that for each odd prime $\ell$ dividing $q$, we have $\alphatil(\ell) \ge 1-2/(\ell-1) \ge 1/2$. Letting $e_1 \coloneqq v_2(q)$ and multiplying the above bound over all the odd primes dividing $q$, we obtain 
$$\frac{\#\mathcal V_q(w)}{\phi(q)^3} \le \frac{\#\mathcal V_{2^{e_1}}(w)}{\phi(2^{e_1})^3} ~ \prod_{\substack{\ell^e \parallel q\\\ell>2}} \frac1{\phi(\ell^e)} \left(1+O\left(\frac1{\ell^{1/2}}\right)\right) \le \frac{\#\mathcal V_{2^{e_1}}(w)}{\phi(2^{e_1})^2} \cdot \phiqrec \exp(O(\sqrt{\log q}))$$
uniformly in coprime residues $w$ mod $q$. Finally, recalling (from the proof of Proposition \ref{prop:etachiBounds}) that the map $v \mapsto v^2+v+1$ is a bijection on the multiplicative group mod $2^{e_1}$, we see that 
\begin{equation}\label{eq:V2e1} 
\#\mathcal V_{2^{e_1}}(w) = \sum_{\substack{a_1, a_2, a_3 \bmod 2^{e_1}\\a_1 a_2 a_3 \equiv w \pmod{2^{e_1}}}} 1 = \sum_{\substack{a_1, a_2 \bmod 2^{e_1}\\ \gcd(a_1 a_2, 2)=1}} 1 = \phi(2^{e_1})^2,
\end{equation}
leading to the bound
$$\frac{\#\mathcal V_q(w)}{\phi(q)^3} \le \phiqrec \exp(O(\sqrt{\log q}))$$
uniformly in coprime residues $w$ mod $q$. Hence this bound also holds true for the ratio $V_q/\phi(q)^3$, and inserting this latter bound into \eqref{eq:Sigma2_ReducntoOrthog} we obtain 
$$\Sigma_2 \ll \frac{x^{1/2} (\log_2 x)^4}{\phi(q) (\log x)^{1-\alphatil\epsilon/4}} \exp(O(\sqrt{\log q})) \ll \frac{x^{1/2}}{\phi(q) (\log x)^{1-\alphatil\epsilon/2}},$$
showing that $\Sigma_2$ is also absorbed into the right hand side of \eqref{eq:SigmaEven_RemInconv} and completing the proof of Theorem \ref{thm:sigmaeven}. \hfill \qedsymbol 

\subsection{Optimality in the restriction $P_6(n)>q$.} \label{subsec:P6Optimal}
We construct a counterexample establishing that the restriction $P_6(n)>q$ is optimal, in the sense that uniformity in $q \le (\log x)^K$ fails if this restriction is weakened or in other words, if the set of inputs $n$ is even slightly enlarged to those having fewer than $6$ prime factors exceeding $q$. To do this, we define 
\begin{equation}\label{eq:Vqtilwdef}
\Vqtil(w) \coloneqq \left\{(v_1, v_2) \in U_q \times U_q: ~ (v_1^2 + v_1 + 1) (v_2^2 + v_2 + 1) \equiv w \pmod q \right\}.   
\end{equation}
We shall first establish that 
\begin{equation}\label{eq:VellTwo916}
\#\VellTwotilwNineSixInv \ge 2\ell^2 \left(1+O\left(\frac1{\sqrt{\ell}}\right)\right)    
\end{equation}
uniformly in primes $\ell \ge 5$, with $16^{-1}$ denoting the multiplicative inverse of $16$ mod $\ell$. For each $(v_1, v_2) \in \VellTwotilwNineSixInv$, we set $a_i \equiv v_i^2+v_i+1 \pmod{\ell^2}$, which is equivalent to $(2v_i+1)^2 \equiv 4a_i-3 \pmod{\ell^2}$. As such, we may write 
\begin{equation}\label{eq:VellTwo_916_FirstBound}
\#\VellTwotilwNineSixInv = \sum_{\substack{(a_1, a_2) \in \UellTwoCart\\a_1 a_2 \equiv 9 \cdot 16^{-1} \pmod{\ell^2}\\\text{each} ~ 4a_i-3\text{ is a square mod }\ell^2}} ~ ~ ~ \sum_{\substack{(v_1, v_2) \in  U_{\ell^2} \times U_{\ell^2}\\\text{each} ~ (2v_i+1)^2 \equiv 4a_i-3 \pmod{\ell^2}}} 1 \ge S_1 + S_2,
\end{equation}
where $S_1$ denotes the contribution of the case $4a_1 - 3 \equiv 4a_2 - 3 \equiv 0 \pmod{\ell^2}$ and $S_2$ denotes the contribution of the case $\ell \nmid (4a_1-3) (4a_2-3)$.

First of all, we see that
\begin{equation}\label{eq:S1Equality}
S_1 = \sum_{\substack{(v_1, v_2) \in \UellTwoCart\\\text{each} ~ (2v_i+1)^2 \equiv 0 \pmod{\ell^2}}} 1 = \sum_{\substack{(v_1, v_2) \in  U_{\ell^2} \times U_{\ell^2}\\\text{each} ~ v_i \equiv -2^{-1} \pmod{\ell}}} 1 = \ell^2.
\end{equation}
We seek to put a lower bound on the sum $S_2$. To do this, we first note that the condition $\ell \nmid (4a_1-3)(4a_2-3)$ in conjunction with the condition that $4a_1 - 3$ and $4a_2-3$ are both squares mod $\ell^2$ are together equivalent to the condition that $\left(\frac{4a_1-3}\ell\right) = \left(\frac{4a_2-3}\ell\right) = 1$; indeed the forward direction is tautological, while the reverse implication is a consequence of Hensel's Lemma. In fact by the same lemma, we see that for each choice of $a_i \in U_\ell^2$ satisfying $\left(\frac{4a_i-3}\ell\right) = 1$, the congruence $t^2 \equiv 4a_i-3 \pmod{\ell^2}$ has exactly two distinct solutions $t$ mod $\ell^2$. If $a_i \not\equiv 1 \pmod\ell$, then $4a_i - 3 \not\equiv 1 \pmod\ell$, so that $t \not\equiv 1 \pmod \ell$ for both of the two aforementioned solutions, and each of them leads to a unique solution $v_i \in U_{\ell^2}$ (given by $2v_i+1 \equiv t \pmod{\ell^2}$). Summarizing our argument, we have shown that 
$$S_2 \ge \sum_{\substack{(a_1, a_2) \in \UellTwoCart\\a_1 a_2 \equiv 9 \cdot 16^{-1} \pmod{\ell^2}\\ \left(\frac{4a_1-3}\ell\right) = \left(\frac{4a_2-3}\ell\right) = 1}} ~ ~ ~ \sum_{\substack{(v_1, v_2) \in  U_{\ell^2} \times U_{\ell^2}\\\text{each} ~ (2v_i+1)^2 \equiv 4a_i-3 \pmod{\ell^2}}} 1 \ge 4 \sum_{\substack{(a_1, a_2) \in \UellTwoCart\\16 a_1 a_2 \equiv 9 \pmod{\ell^2}\\\text{each} ~ a_i \not\equiv 1 \pmod \ell, ~ \left(\frac{4a_i-3}\ell\right) = 1}} 1.$$
Now the condition $a_1 a_2 \equiv 9 \cdot 16^{-1} \pmod{\ell^2}$ shows that if $a_1 \equiv 1 \pmod \ell$, then $a_2 \equiv 9 \cdot 16^{-1} \pmod{\ell}$ which can be lifted to a residue class mod $\ell^2$ in at most $\ell$ ways. This shows that ignoring the condition $``\text{each} ~ a_i \not\equiv 1 \pmod \ell$" in the last sum in the above display incurs an error of $O(\ell)$. We deduce that
$$S_2 \ge 4 \sum_{\substack{(a_1, a_2) \in \UellTwoCart\\16 a_1 a_2 \equiv 9 \pmod{\ell^2}\\\left(\frac{4a_1-3}\ell\right) = \left(\frac{4a_2-3}\ell\right) = 1}} 1 + O(\ell).$$
Moreover, for any $a_i \in U_{\ell^2}$ satisfying $\left(\frac{4a_i-3}\ell\right) = 1$, we can write $4 a_i - 3$ in the form $u_i^2 + \ell c_i$ $\pmod{\ell^2}$ for some $u_i, c_i \in \{0, 1, \dots, \ell-1\}$ such that $\gcd(u_i, \ell)=1$. In fact, given $a_i$, there are exactly two possible choices of $u_i$ and exactly one possible choice of $c_i$ (this is because $u_i$ can only be one of the two square roots of $4a_i-3$ mod $\ell$, and either of them determines the same value of $c_i$ via the congruence $c_i \equiv \frac{4a_i-3-u_i^2}\ell$ mod $\ell$). Hence
\begin{equation}\label{eq:S2PreRemoval}
S_2 \ge \sum_{\substack{(u_1, u_2) \in U_\ell \times U_\ell\\(u_1^2+3)(u_2^2+3) \equiv 9 \pmod{\ell}}} ~ \sum_{\substack{(c_1, c_2) \in \Z/\ell\Z \times \Z/\ell\Z\\c_1 (u_2^2+3) + c_2 (u_1^2+3) \equiv \frac{9-(u_1^2+3)(u_2^2+3)}\ell \pmod{\ell}}} 1 + O(\ell),
\end{equation} 
where we have noted that since $4a_i \equiv u_i^2 + 3 + \ell c_i \pmod{\ell^2}$, the condition $(u_1^2 + 3 + \ell c_1) (u_2^2 + 3 + \ell c_2) \equiv 16 a_1 a_2 \equiv 9 \pmod{\ell^2}$ can be rewritten as $\ell(c_1 (u_2^2+3) + c_2 (u_1^2+3)) \equiv 9-(u_1^2+3)(u_2^2+3) \pmod{\ell^2}$. Now given $c_1$, the congruence involving $c_1$ and $c_2$ in \eqref{eq:S2PreRemoval} determines $c_2$ uniquely mod $\ell$. Varying $c_1$ over the $\ell$ possibilities, we thus find that 
\begin{equation}\label{eq:S2PostRemoval}
S_2 \ge \ell\sum_{\substack{(u_1, u_2) \in U_\ell \times U_\ell\\(u_1^2+3)(u_2^2+3) \equiv 9 \pmod{\ell}}} 1 + O(\ell) = \ell\sum_{\substack{(u_1, u_2) \in \F_\ell \times \F_\ell\\(u_1^2+3)(u_2^2+3) = 9 \text{ in }\F_{\ell}}} 1 + O(\ell),
\end{equation}
where in the last equality above, we have noted that there is exactly one possible tuple $(u_1, u_2) \in  \F_\ell \times \F_\ell$ satisfying $(u_1^2+3)(u_2^2+3) = 9 \text{ in }\F_{\ell}$, in which either $u_1$ or $u_2$ is zero (namely the tuple $(u_1, u_2) = (0, 0)$). 

In order to estimate the last sum in \eqref{eq:S2PostRemoval}, we proceed in a manner similar to the proof of Theorem 1.4(b) in \cite{PSR23}: we first show that the polynomial $G(X, Y) \coloneqq (X^2+3)(Y^2+3)-9$ is absolutely irreducible over $\F_\ell[X, Y]$. \footnote{While this claim follows from the absolute irreducibility established in the proof of \cite[Theorem 1.4(b)]{PSR23}, we shall give a more straightforward self-contained argument that suffices for our current setting.} Indeed, assume that $G = U V$ for some $U, V \in \Fellbar[X, Y]$; we wish to show that one of $U$ or $V$ must be constant. If either $U$ or $V$ is a polynomial only in $Y$ (say $U(X, Y) = u(Y)$), then taking $X$ to be a root $\theta \in \Fellbar$ of the polynomial $X^2+3$ on both sides of the identity $G=UV$, we obtain $-9 = G(\theta, Y) = u(Y) V(\theta, Y)$ in the ring $\Fellbar[Y]$, showing that $U(X, Y) = u(Y)$ must be constant. On the other hand, if neither $U$ nor $V$ is a polynomial in $Y$ only, then by comparing the degrees in the variable $X$ on both sides of the identity $H = UV$, we find that $U(X, Y) = u_1(Y) X + u_0(Y)$ and $V(X, Y) = v_1(Y) X + v_0(Y)$ for some $u_i, v_i \in \Fellbar[Y]$. Comparing the coefficients of $X$ on both sides of the identity
$$(Y^2+3)(X^2+3) - 9 = (u_1(Y) X + u_0(Y))(v_1(Y) X + v_0(Y)),$$
we get the three identities $u_1(Y) v_1(Y) = Y^2+3$, $u_1(Y) v_0(Y) + u_0(Y) v_1(Y) = 0$ and $u_0(Y) v_0(Y) = 3Y^2$. Again, letting $\theta \in \Fellbar$ be a root of the polynomial $Y^2+3$, the first of the three identities shows that $Y-\theta$ divides exactly one of $u_1$ or $v_1$ (as the polynomial $Y^2+3$ is separable over $\F_\ell$). Assuming without loss of generality that $(Y-\theta) \mid u_1(Y)$ (so that $\gcd(Y-\theta, v_1(Y))=1$), the second of the aforementioned identities forces $Y-\theta$ to divide $u_0(Y)$, leading to a contradiction in the third identity (since $3 \theta^2 = -9 \ne 0$ in $\Fellbar$). This establishes that $G$ is indeed absolutely irreducible over $\F_\ell[X, Y]$. 

Consequently, the variant of the Weil bound established in \cite[Corollary 2(b)]{LY94} yields, from \eqref{eq:S2PostRemoval},
$$S_2 \ge \ell(\ell + O(\sqrt{\ell})) + O(\ell) = \ell^2 \left(1+O\left(\frac1{\sqrt{\ell}}\right)\right).$$
Combining this with \eqref{eq:S1Equality} and \eqref{eq:VellTwo_916_FirstBound} completes the proof of \eqref{eq:VellTwo916}. 

Now set $q \coloneqq 2 \left(\prod_{5 \le \ell \le Y} \ell\right)^2$, where $Y \ll \log_2 x$ is a parameter to be chosen appropriately later. Then $q \le (\log x)^{O(1)}$ and letting $w_q$ denote the unique coprime residue mod $q$ satisfying $w_q \equiv 9 \cdot 16^{-1} \pmod{\ell^2}$ for all primes $\ell \in [5, Y]$, the lower bound \eqref{eq:VellTwo916} yields 
\begin{equation}\label{eq:VqwqLB}
\#\Vqtil(w_q) = \prod_{5 \le \ell \le Y} \#\VellTwotilwNineSixInv \gg 2^{\pi(Y)} q \prod_{5 \le \ell \le Y} \left(1+O\left(\frac1{\sqrt{\ell}}\right)\right) \gg 2^{\pi(Y)} q \exp(-C_2 \sqrt{\log_2 x}).
\end{equation}
for some absolute constant $C_2>0$. Now consider any integer $n$ of the form $P_1^2P_2^2$, where $P_1$ and $P_2$ are primes satisfying $x^{1/10} < P_2 \le x^{1/6} < P_1 \le x^{1/2}/P_2$ and $(P_1, P_2)$ mod $q$ $\in \Vqtil(w_q)$. Then $n \le x$, $P_4(n) = P_2 > x^{1/10} > q$ and $\sigma(n) = (P_1^2 + P_1 + 1) (P_2^2+P_2+1) \equiv w_q \pmod q$. By the Siegel-Walfisz Theorem and partial summation, we obtain 
\begin{equation*}
\begin{split}
\sum_{\substack{n \le x: ~ P_4(n)>q\\ \sigma(n) \equiv w_q \pmod q}} 1 &\ge \sum_{(v_1, v_2) \in \Vqtil(w_q)} ~ \sum_{\substack{x^{1/10} < P_2 \le x^{1/6}\\P_2 \equiv v_2 \pmod q}}  ~ \sum_{\substack{x^{1/6} < P_1 \le x^{1/2}/P_2\\P_1 \equiv v_1 \pmod q}} 1\\ &\gg \sum_{(v_1, v_2) \in \Vqtil(w_q)} ~ \sum_{\substack{x^{1/10} < P_2 \le x^{1/6}\\P_2 \equiv v_2 \pmod q}} \frac{x^{1/2}/P_2}{\phi(q) \log x} \gg \frac{\#\Vqtil(w_q)}{\phi(q)^2} \cdot \frac{x^{1/2}}{\log x}.
\end{split}
\end{equation*}
An application of \eqref{eq:VqwqLB} now yields
$$\sum_{\substack{n \le x: ~ P_4(n)>q\\ \sigma(n) \equiv w_q \pmod q}} 1 ~ \gg ~ \frac{2^{\pi(Y)}}{\phi(q)} \cdot \frac{x^{1/2}}{\log x} \exp(-C_2 \sqrt{\log_2 x})$$
for some constant $C_2>0$. By Lemma \ref{lem:sigmaqcoprimEvenq}, the quantity on the right hand side above grows strictly faster than the expected main term $\phiqrec \sum_{\substack{n \le x\\ (\sigma(n), q)=1}} 1$ as soon as $2^{\pi(Y)} > (\log x)^{(1+\delta) \alphatil}$ for some fixed $\delta>0$, which in turn is equivalent to $\pi(Y) > (1+\delta) \alphatil \log_2 x/\log 2$. But now $\pi(Y) > Y/2\log Y$, while
\begin{equation}\label{eq:alphatilq_Countereg}
\begin{split}
\alphatilq = \prod_{\substack{5 \le \ell \le Y\\\ell \equiv 1 \pmod 3}} \left(1-\frac2{\ell-1}\right) = \exp\left(-2\sum_{\substack{5 \le \ell \le Y\\\ell \equiv 1 \pmod 3}} \frac1\ell + O(1) \right) < \frac{K_2}{\log Y} 
\end{split}
\end{equation}
for some absolute constant $K_2>0$, where we have used the prime number theorem in arithmetic progressions to estimate the sum on $\ell$. As such, the desired condition $\pi(Y) > (1+\delta) \alphatil \log_2 x/\log 2$ holds as soon as $Y> 2K_2 (1+\delta) \log_2 x/\log 2$, which is compatible with the only other condition $Y \ll \log_2 x$ needed on the parameter $Y$. Choosing $Y$ accordingly, we have therefore established that the condition $P_6(n)>q$ in Theorem \ref{thm:sigmaeven} cannot be weakened to $P_4(n)>q$ in the range of uniformity in $q$. Since the largest odd divisor of $n$ is a perfect square, it follows that the restriction $P_6(n)>q$ in Theorem \ref{thm:sigmaeven} is indeed optimal.

\section{Distribution of the sum-of-divisors function to squarefree even moduli: Proof of Theorem \ref{thm:sigmaevensqfree}}
As in the beginning of the previous section, we can show that the right hand sides of \eqref{eq:sigmaevensqfree} and \eqref{eq:SigmaEvenConvenient} are equal up to a negligible error: indeed, by previous arguments, this only needs the following analogue and consequence of \eqref{eq:P6nley_Bound}:
$$\sum_{\substack{n \le x: ~ P_4(n) \le y\\ P(n)>z; ~ p>y \implies p^4 \nmid n\\\gcd(\sigma(n), q)=1}} 1 \ll \frac{x^{1/2}}{(\log x)^{1-\alphatil \epsilon/2}}.$$
Hence to complete the proof of the theorem, it suffices to show the following analogue of \eqref{eq:SigmaEven_RemInconv} uniformly in squarefree even moduli $q \le (\log x)^K$ and in coprime residues $a$ mod $q$:
\begin{equation}\label{eq:SigmaEvenSqfree_RemInconv}
\sum_{\substack{n \le x: ~ q<P_4(n) \le y\\ P(n)>z; ~ p>y \implies p^4 \nmid n\\\sigma(n)\equiv a \pmod q}} 1 \ll \frac{x^{1/2}}{\phi(q) (\log x)^{1-\alphatil \epsilon/2}}.
\end{equation}
As before, we write the sum on the left hand side in the form $\SigmatilOne + \SigmatilTwo$, where $\SigmatilOne$ denotes the contribution of the $n$ which are divisible by the fourth power of a prime exceeding $q$. Then with $\Sigma_1$ as defined in the previous section, it follows by the intermediate bounds in \eqref{eq:SigmaOneBound} that 
$$\SigmatilOne \le \sum_{\substack{n \le x: ~ P_4(n) \le y\\ P(n)>z; ~ p>y \implies p^4 \nmid n\\\exists p>q: ~ p^4|n\\\sigma(n)\equiv a \pmod q}} 1 \le \sum_{\substack{n \le x: ~ P_6(n) \le y\\ P(n)>z; ~ p>y \implies p^4 \nmid n\\\exists p>q: ~ p^4|n\\\sigma(n)\equiv a \pmod q}} 1 \ll \frac{x^{1/2}}{q (\log x)^{1-\alphatil\epsilon/2}} $$
is absorbed in the right hand side of \eqref{eq:SigmaEvenSqfree_RemInconv}.

In order to estimate the sum $\Sigma_2$, we note that any $n$ counted in this sum has $P_4(n)>q$ but is not divisible by the fourth power of a prime exceeding $q$. Consequently, as in the previous section, we may write $n = 2^k r^2 (P_1 P_2)^2$ where $k \ge 0$, $P_2(r) \le y$, and $P_1$, $P_2$ are primes satisfying $P_1 = P(n)>z$, $q<P_2<P_1$, and $\sigma(n) = \sigma(2^k) \sigma(r^2) \prod_{j=1}^2 (P_j^2 + P_j + 1)$. Given $k$ and $r$, the congruence $\sigma(n) \equiv a \pmod q$ forces $(P_1, P_2)$ mod $q$ $\in \Vqtil(a\sigma(2^k r^2)^{-1})$, where $\Vqtil(w)$ is as defined in \eqref{eq:Vqtilwdef} for any coprime residue $w$ mod $q$. Hereafter, setting $\widetilde V_q \coloneqq \max\{\#\Vqtil(w): ~ w \in U_q\}$ and replicating the arguments leading to \eqref{eq:Sigma2_ReducntoOrthog} gives the following analogue of the latter bound:
\begin{equation}\label{eq:Sigma2_ReducntoLY}
\SigmatilTwo \ll \frac{\widetilde V_q}{\phi(q)^2} \cdot \frac{x^{1/2} (\log_2 x)^3}{(\log x)^{1-\alphatil\epsilon/4}} \expOlogtwoqOOne.   
\end{equation}
Now since $q$ is squarefree, we may write $\#\Vqtil(w) = \prodellq \#\Velltilw$ for any coprime residue $w$ mod $q$. Here $\#\Velltilw$ is no more than the number of $\F_\ell$-rational points of the polynomial $H(X, Y) \coloneqq (X^2+X+1)(Y^2+Y+1)-w$. We claim that this latter number is no more than $\phi(\ell)(1+O(\ell^{-1/2}))$. By a computation similar to \eqref{eq:V2e1}, this is true for $\ell=2$ (without the multiplicative error term), so we may consider the case $\ell \ge 5$. But in fact, an argument entirely analogous to that given for the polynomial $(X^2+3)(Y^2+3)-9$ in subsection \ref{subsec:P6Optimal}, shows that the polynomial $H$ is absolutely irreducible over $\F_\ell[X, Y]$. 
(Here again, it is important that $w \ne 0 \in \Fellbar$ and that since $\ell \ge 5$, the polynomial $Y^2+Y+1$ is separable over $\F_\ell$.) As such, \cite[Corollary 2(b)]{LY94} establishes our claim. 

As a consequence, we obtain $\#\Vqtil(w) \le \prod_{\ell \mid q} \phi(\ell) (1+O(\ell^{-1/2})) \le \phi(q) \exp(O(\sqrt{\log q}))$ uniformly in coprime residues $w$ mod $q$. The same bound thus continues to hold for $\widetilde{V}_q$, and \eqref{eq:Sigma2_ReducntoLY} shows that $\SigmatilTwo$ is also absorbed in the right hand side of \eqref{eq:SigmaEvenSqfree_RemInconv}, completing the proof of the theorem.
\subsection{Optimality in the restriction $P_4(n)>q$.} \label{subsec:P4Optimal}
The restriction $P_4(n)>q$ is crucial and optimal in the sense that weak equidistribution fails (in the range of uniformity in $q$ provided in the theorem) as soon as one enlarges the set of inputs $n$ to those having fewer prime factors exceeding $q$. Indeed, let $q \coloneqq 2 \prod_{5 \le \ell \le Y} \ell$ for some parameter $Y \ll \log_2 x$ to be chosen appropriately. For a prime $P \in (q, x^{1/2}]$, the congruence $\sigma(P^2) \equiv 3 \pmod q$ holds for $P$ lying in exactly $2^{\omega(q)-1}$ distinct coprime residue classes modulo $q$ (namely those lying in the residue classes $-2$ or $1$ modulo each of the odd prime divisors of $q$). As such, by the Siegel-Walfiz theorem, there are $\gg \frac{2^{\omega(q)}}{\phi(q)} \frac{x^{1/2}}{\log x}$ many integers $n \le x$ having $P_2(n)>q$ and $\sigma(n) \equiv 3 \pmod q$, coming only from the squares of the primes lying in the interval $(x^{1/4}, x^{1/2}]$. By \eqref{eq:sigman_evenq_estim}, the coprime residue class $3$ mod $q$ will be over-represented as soon as $2^{\omega(q)} > (\log{x})^{(1+\delta)\alphatil(q)}$ for a fixed $\delta>0$. By the same computation as in \eqref{eq:alphatilq_Countereg}, we have $\alphatil(q) \ll 1/\log Y$, whereas $\omega(q) \ge Y/2\log Y$. The inequality $2^{\omega(q)} > (\log{x})^{(1+\delta)\alphatil(q)}$ is thus ensured as soon as we choose $Y>K_1 \log_2 x$ for some appropriate constant $K_1>0$, a condition that is consistent with the only requirement $Y \ll \log_2 x$ on $Y$. This shows that the restriction $P_2(n)>q$ is inadequate to get weak equidistribution to moduli varying up to a fixed arbitrary power of $\log x$. Since $n$ is of the form $2^k m^2$ for some odd $m$, it follows that the restriction $P_4(n)>q$ in Theorem \ref{thm:sigmaeven} is optimal. 

\begin{rmk}
The above example may be compared with the one given in the discussion following the statement of Theorem 1.3 in \cite{PSR23}.
\end{rmk}
\section*{Acknowledgements}
This work was done in partial fulfillment of my PhD at the University of Georgia. As such, I would like to thank my advisor, Prof.\@ Paul Pollack, for useful discussions and valuable comments. I would also like to thank the Department of Mathematics at UGA for their support and hospitality. 

\providecommand{\bysame}{\leavevmode\hbox to3em{\hrulefill}\thinspace}
\providecommand{\MR}{\relax\ifhmode\unskip\space\fi MR }
\providecommand{\MRhref}[2]{%
  \href{http://www.ams.org/mathscinet-getitem?mr=#1}{#2}
}
\providecommand{\href}[2]{#2}

\end{document}